\documentclass[11pt,oneside,reqno]{amsart}
\usepackage[latin1]{inputenc}
\usepackage[english]{babel}
\usepackage{mathrsfs}
\usepackage{indentfirst}
\usepackage{paralist}
\usepackage{amssymb}
\usepackage{amsthm}
\usepackage{amsmath}
\usepackage{amscd}
\usepackage{graphicx}
\usepackage[colorlinks=true]{hyperref}
\usepackage[margin=1in]{geometry}
\usepackage{cite}
\usepackage{lineno} 

\usepackage{xcolor}
\definecolor{ForestGreen}{rgb}{0.13, 0.55, 0.13}

\theoremstyle{plain} 
\newtheorem{theorem}{Theorem}[section] 
\newtheorem*{theorem*}{Main Result}
\newtheorem{thm*}{Known result}
\newtheorem{corollary}[theorem]{Corollary} 
\newtheorem{lemma}[theorem]{Lemma}
\newtheorem{proposition}[theorem]{Proposition}
\newtheorem{definition}[theorem]{Definition}

\theoremstyle{definition}

\theoremstyle{remark}
\newtheorem{remark}[theorem]{Remark}

\numberwithin{equation}{section}

\newcommand{\eqlab}[1]{\begin{equation}  \begin{aligned}#1 \end{aligned}\end{equation}} 
\newcommand{\bgs}[1]{\begin{equation*} \begin{aligned}#1\end{aligned}\end{equation*}} 
\newcommand{\sys}[2][]{\begin{equation*}#1  \left\{\begin{aligned}#2\end{aligned}\right.\end{equation*}}

\newcommand{\R}{\ensuremath{\mathbb{R}}}
\newcommand{\Rn}{\ensuremath{\mathbb{R}^n}}
\newcommand{\N}{\ensuremath{\mathbb{N}}}
\newcommand{\eps}{\ensuremath{\varepsilon}}

\newcommand{\Co}{\mathcal C}
\newcommand{\E}{\mathcal E}

\newcommand{\G}{\mathcal G}

\newcommand{\W}{\mathcal W}
\DeclareMathOperator{\Per}{Per}
\DeclareMathOperator{\loc}{loc}
\DeclareMathOperator{\diam}{diam}
\DeclareMathOperator{\dist}{dist}
\DeclareMathOperator{\Tail}{Tail}

\renewcommand{\le}{\leqslant}
\renewcommand{\leq}{\leqslant}
\renewcommand{\ge}{\geqslant}
\renewcommand{\geq}{\geqslant}
\renewcommand{\emptyset}{\varnothing}

\title[Minimisers of a fractional seminorm and nonlocal  minimal surfaces]{Minimisers of a fractional seminorm \\ and nonlocal  minimal surfaces}

\thanks{
{\em Claudia Bucur}: Istituto Nazionale di Alta Matematica,
Piazzale Aldo Moro 5,
00185 Rome, Italy, and
Dipartimento di Matematica, Universit\`a di Milano,
Via Saldini 50, 20133 Milan, Italy. {\tt claudia.bucur@aol.com}
\\
{\em Serena Dipierro}:
Department of Mathematics
and Statistics,
University of Western Australia,
35 Stirling Hwy, Crawley WA 6009, Australia.
{\tt serena.dipierro@uwa.edu.au}\\
{\em Luca Lombardini}: Department of Mathematics
and Statistics,
University of Western Australia,
35 Stirling Hwy, Crawley WA 6009, Australia.
{\tt luca.lombardini@uwa.edu.au}\\
{\em Enrico Valdinoci}:
Department of Mathematics
and Statistics,
University of Western Australia,
35 Stirling Hwy, Crawley WA 6009, Australia. {\tt enrico.valdinoci@uwa.edu.au}\\
The authors are members of INdAM and AustMS.
The first author is supported by the INdAM Starting Grant
``PDEs, free boundaries, nonlocal equations and applications''.
The second and fourth authors
are supported by the Australian Research Council
Discovery Project DP170104880 NEW ``Nonlocal Equations at Work''.
The second author is supported by
the Australian Research Council DECRA DE180100957
``PDEs, free boundaries and applications''. Part of this work was carried
out during a very pleasant and fruitful visit of the first author to the
University of Western Australia, which we thank for the warm hospitality.}

\author[C. Bucur]{Claudia Bucur}
\author[S. Dipierro]{Serena Dipierro}
\author[L. Lombardini]{Luca Lombardini}
\author[E. Valdinoci]{Enrico Valdinoci}

\begin{document}

\begin{abstract}
The recent literature has intensively studied two classes of
nonlocal variational problems, namely the ones related to the minimisation
of energy functionals that act on functions in suitable Sobolev-Gagliardo
spaces, and the ones related to the minimisation
of fractional perimeters that act on measurable sets
of the Euclidean space.

In this article, we relate these two types of variational problems.
Specifically,
we investigate the connection between the nonlocal minimal surfaces
and the minimisers of the $W^{s,1}$-seminorm.

In particular, we show that a function is a
minimiser for the fractional seminorm if and only if its level sets
are minimisers for the fractional perimeter, and
that the characteristic function of a nonlocal minimal surface
is a minimiser for the fractional seminorm; we also provide an
existence result for minimisers of the fractional seminorm,
an explicit non-uniqueness example
for nonlocal minimal surfaces, and
a Yin-Yang result describing the full and void patterns of nonlocal minimal surfaces.
\end{abstract}
\maketitle

\tableofcontents

\begin{quote}
{\tt\em {``Yang stands for destruction; Yin stands for conservation. Yang brings about disintegration;
Yin gives shape to things.''} \tt
The Yellow Emperor's Classic of Medicine.}\end{quote}\bigskip\bigskip

\section{Introduction}

In the recent literature, a number of variational problems
related to nonlocal functionals of fractional type have been intensively
studied.
Some of these problems deal with {\em functions} minimising
suitable fractional seminorms of Sobolev-Gagliardo-Slobodecki\u{\i} type,
see e.g.~\cite{MR2498561, MR2707618, MR3035063, MR3165278, MR3503212, MR3542614, MR3630640}.
Other problems focus on {\em sets} minimising a fractional notion
of perimeter. One of the main goals of this paper is
to connect the minimisation notion for functions to that for sets,
see e.g.~\cite{nms, MR3322379, MR3412379}.
\medskip

More specifically, in this article we investigate the relation between
the minimisers of the fractional perimeter and those of a $W^{s,1}$-seminorm.
A general notion of minimisers will be introduced to conveniently highlight the connections
between these two problems and establish a suitable existence theory.\medskip

The main results obtained can be grouped into five classes:
\begin{enumerate}
\item {\em Equivalence results:} we will show that a function is a
minimiser for the fractional seminorm if and only if its level sets
are minimisers for the fractional perimeter;
\item {\em Minimising properties of nonlocal minimal surfaces:}
we will show that the characteristic function of a nonlocal minimal surface
is a minimiser for the fractional seminorm
(among all possible competitors, not only among characteristic functions);
\item {\em Existence results:} we will utilise the level sets method to obtain
existence of minimisers for the fractional seminorm;
\item {\em Non-uniqueness results:} we will provide an explicit example
of external datum which leads to at least two nonlocal minimal surfaces
(thus showing  that the level sets method has necessarily
to take into account this ``pathological''
 possibility of 
lack of uniqueness, and cannot be further simplified);
\item {\em Yin-Yang results:} we will show that if the external datum
is void (respectively, full)
in the vicinity of a given domain, then the corresponding nonlocal
minimal surface is necessarily void (respectively, full) inside the domain --
even if the external datum is completely full (respectively, void) at infinity.
\end{enumerate}

Some of these results can also be seen as a nonlocal counterpart of
the classical works in~\cite{bomby, SWZ}, where the minimisers of the perimeter functional have
been related to the minimisers of the $W^{1,1}$-seminorm, also in view of mean curvature type
equations. As we will see in the course of the proofs, dealing with
the nonlocal interactions requires in our setting different approaches than in the classical case.\medskip

Some of these results rely on a suitable variation of a well established
fractional co-area formula
and an auxiliary fractional Hardy inequality
with optimal exponents, but several technical difficulties have also
to be taken into account, especially to deal with possibly unbounded
external data, and to exchange different concepts of minimisations
(e.g., among sets and among functions) which, in principle,
are not clearly related to each other.

Moreover, an interesting feature of the set of problems
addressed in this paper is that there is a strong interplay
between the functional analysis aspects of the setting
and the geometric one: for instance, the Yin-Yang result,
which is intrinsically geometric in spirit, plays an important role
in establishing other merely analytical results.

Let us also stress that geometric results such as the Yin-Yang
one that we propose fit in the research trend of describing
qualitatively the nonlocal minimal surfaces,
with the aim of understanding similarities and differences
with respect to their classical counterparts: in this sense,
these results are not just nice and folkloristic remarks, but,
given the inaccessibility offered so frequently by nonlocal minimal surfaces
to classical analytic methods, they often turn out to be one of the few
tools to deeply understand these new and complicated situations
and on many occasions turn out to be a cornerstone for
a solid development of our knowledge on this subject.
\medskip

The precise formulation of our problems and the specific results obtained
are outlined in the forthcoming subsections.

\subsection{A general minimisation problem related to the fractional $1$-Laplacian,
and its relation to the fractional perimeter functional}

We now introduce a minimisation problem related to the ``fractional $1$-Laplacian''.
The rough idea behind our formal definition is the following.
For $p\ge1$ and~$s\in(0,1)$, the
``fractional $p$-Laplacian'' is the operator obtained
from the minimisation of the $W^{s,p}$-seminorm,
see e.g.~\cite{MR3148135, MR3307955, MR3542614, MR3861716, MR3886598}.
A natural condition for such a minimisation consists in prescribing
the functions outside a given domain~$\Omega$, and thus
the energy to be minimised consists in the terms
of the $W^{s,p}(\R^n)$-seminorm which have a
nontrivial interaction
with~$\Omega$: that is, one can define
\[ \Co \Omega:=\R^n\setminus\Omega\qquad{\mbox{and}}\qquad
 	Q(\Omega):=\R^{2n}\setminus (\Co \Omega)^2,
	\]
and consider the minimisation problem of the energy functional
\begin{equation}\label{GAG:0}
\frac12 \iint_{Q(\Omega)} \frac{|u(x) -u(y)|^p }{|x-y|^{n+sp}}
		 dx\, dy.\end{equation}
We will focus specifically on the case~$p=1$, thus reducing~\eqref{GAG:0}
to the energy functional
\begin{equation}\label{GAG}
		\G (u, \Omega) := \frac12\iint_{Q(\Omega)} 
	\frac{|u(x) -u(y)|}{|x-y|^{n+s}}dx\, dy.
\end{equation}
Interestingly, while on the one hand, it is natural to consider the case~$p=1$, given also its relation to the 
BV-seminorm in the limit as~$s\nearrow1$
(see~\cite{BBM, davila}), on the other hand, it is a
challenging case to take into account, due to the loss of convexity of the nonlocal operator.

Furthermore, in order to allow more general external data,
 rather than the functional in~\eqref{GAG} we will consider
its oscillation with respect to a given
function.

The precise mathematical details of this formulation go as follows.
Throughout the entire paper,
we consider $\Omega \subset \Rn$  to be a  bounded, open set with Lipschitz boundary (unless otherwise specified), and $s\in (0,1)$ to be a fixed number,
and we define the functional space
 \[
	 \W^{s}(\Omega):= \big\{ u\colon \Rn \to \R \mbox{ measurable}\; | \; u \in W^{s,1}(\Omega) \big\}.
 \]
We point out that~$\W^{s}(\Omega)$ is the space of the functions
{\em defined in the whole of~$\Rn$} whose restriction to~$\Omega$
belongs to~$W^{s,1}(\Omega)$. We recall that a function $u:\Omega\to\R$ belongs to the fractional Sobolev space~$W^{s,1}(\Omega)$ if $u\in L^1(\Omega)$ and
\[
[u]_{W^{s,1}(\Omega)}:=\int_\Omega\int_\Omega \frac{|u(x)-u(y)|}{|x-y|^{n+s}}\,dx\,dy<+\infty.
\]
The space~$W^{s,1}(\Omega)$ is a Banach space with respect to the norm
\[
\|u\|_{W^{s,1}(\Omega)}:=\|u\|_{L^1(\Omega)}+[u]_{W^{s,1}(\Omega)}.
\]
\medskip

In this setting, we introduce a general notion of minimisers.

\begin{definition}\label{defn11}
	We say that $u\in \W^s(\Omega)$ is an $s$-minimal function in $\Omega$ if
	\begin{equation}\label{GH2}
	\iint_{Q(\Omega)} \big(|u(x) -u(y)|-|v(x)  - v(y)|\big) 
	\frac{dx\, dy}{|x-y|^{n+s}} \leq 0,
	\end{equation}
	for any competitor $v\in \W^s(\Omega)$ such that $u=v$ almost everywhere in $\Co \Omega$.
	
\end{definition}

We observe that the setting in Definition~\ref{defn11}
is well posed, thanks to the fractional Hardy inequality in~\cite{teolu}. 
Precisely, for any $u\in W^{s,1}(\Omega)$, 
it holds that
 	\[
		\int_{\Omega} \frac{ |u(x)|}{\big(\!\dist(x,\partial\Omega) \big)^{s}}\, dx \leq C \|u\|_{W^{s,1}(\Omega)},
	\] 
	for some $C=C(n,s,\Omega)>0$,
 which implies that 
	\eqlab{\label{frch}
	\int_{\Omega} dx \left(\int_{\Co \Omega} \frac{|u(x)|}{|x-y|^{n+s}} \, dy\right) \leq  C \|u\|_{W^{s,1}(\Omega)},
	}
	 see~\cite[Theorem D.1.4, Corollary D.1.5]{tesilu}.
Then, given $u,v \in \W^s(\Omega)$ such that $u=v$ almost everywhere in $\Co \Omega$, it follows that
\begin{equation}\label{GHE}
 	\iint_{Q(\Omega)} \big||u(x) -u(y)|-|v(x)  - v(y)|\big| \frac{dx\, dy}{|x-y|^{n+s}} \leq C \| u-v\|_{W^{s,1}(\Omega)}.	
\end{equation}
	In particular,
the left hand side of~\eqref{GH2} is finite, in light of~\eqref{GHE}.\medskip

We remark that the minimisation setting in Definition~\ref{defn11}
comprises the energy minimisation of~\eqref{GAG} as a particular case,
and the results that we present here
are actually new even in the
simpler mathematical formulation given by the minimisers of~\eqref{GAG}
(a more precise comparison between
the setting in Definition~\ref{defn11} and the minimisers of~\eqref{GAG}
will be presented in Lemma~\ref{TGAet}).
In any case, we think that the setting in Definition~\ref{defn11}
is more convenient, since it does not need to require the finiteness
of the full energy contributions, thus allowing more general external data
(even though, from the technical point of view,
the simpler formulation in~\eqref{GAG} 
better suits the direct methods of the calculus of variations, as we
will highlight in Theorem~\ref{APP:TH:E}).\medskip

One of the main objectives of this paper is to relate
the minimisation setting of fractional $1$-Laplace type functionals,
as introduced in Definition~\ref{defn11}, to the minimisers of the fractional
perimeter functional, as introduced in~\cite{nms}.\medskip

To this end, in the light of~\cite{nms},
we recall that, given~$s\in(0,1)$, the $s$-fractional perimeter of a measurable set $E \subset \Rn$ in an open set $\Omega\subset\R^n$ is defined as
    \begin{equation}\label{PER}
     \Per_s(E,\Omega): = \frac12\iint_{Q(\Omega)} \frac{|\chi_{E}(x) -\chi_E(y)|}{|x-y|^{n+s}} \, dx \, dy.
     \end{equation}

As customary, one considers ``nonlocal minimal surfaces'',
i.e. minimisers of this fractional perimeter, according to the following setting:         

    \begin{definition}\label{PERMI}
We say that $E\subset \Rn$ is an $s$-minimal set in $\Omega$ if 
	$\Per_s(E,\Omega)<+\infty$ 
		and 
	\[ 
		\Per_s(E,\Omega) \leq \Per_s(F, \Omega) \quad \mbox{ for any } \quad F \subset \Rn \quad \mbox{ such that } \quad F\setminus \Omega=E \setminus \Omega .
		\]
		Given a set $E_0\subset\R^n$, we say that $E\subset \Rn$ is an $s$-minimal set in $\Omega$ with respect to $E_0$ if $E$ is $s$-minimal in $\Omega$ and $E\setminus\Omega=E_0\setminus\Omega$.
\end{definition}

On the one hand, we notice that the minimisation properties in Definitions~\ref{defn11}
and~\ref{PERMI} are, in principle, structurally different,
since Definition~\ref{defn11} deals with the minimisation 
among {\em functions},
while Definition~\ref{PERMI} deals with the minimisation 
among {\em sets}
(in particular, the two definitions cannot be mutually confused, since they
refer to different objects).

On the other hand, one of the main results of this paper consists in a suitable ``equivalence
between the minimisation properties in Definitions~\ref{defn11}
and~\ref{PERMI}''. Namely, a function~$u\in \W^s(\Omega)$ is $s$-minimal
according to Definition~\ref{defn11} if and only if its level sets are $s$-minimal
according to Definition~\ref{PERMI}. In this sense,
the notion of $s$-minimal function generalises the one
of $s$-minimal set. The precise result goes as follows.

\begin{theorem}\label{STER}
If $u\in \W^s(\Omega)$ is an $s$-minimal function in~$\Omega$, then,
for all $\lambda \in \R$,
the set~$\{  u \geq \lambda \}$ is $s$-minimal in $\Omega$.

Viceversa,
if~$u\in \W^s(\Omega)$ and~$\{ u \geq \lambda \}$ is
an $s$-minimal set in $\Omega$ for almost every $\lambda \in \R$,
then $u$ is an $s$-minimal function in~$\Omega$.
\end{theorem}

Theorem~\ref{STER} here can be seen as a counterpart in the fractional setting of
a classical result in~\cite{bomby, SWZ} which relates the minimisers of a $1$-Laplace type
functional with the minimisers of the classical perimeter functional.
See also~\cite{MR3263922} for classical results relating functions
with least gradient and solutions of $1$-Laplace equations.
\medskip

The close relationship between the minimisation problems in Definitions~\ref{defn11}
and~\ref{PERMI} is also highlighted by the fact that
characteristic functions of 
$s$-minimal sets (according to Definition~\ref{PERMI})
are $s$-minimal functions (according to Definition~\ref{defn11}), 
with respect to all competitors in~$\W^s(\Omega)$
and not only with respect to characteristic functions. 
We give the precise result in the following theorem.

\begin{theorem}\label{TUTT}
Let $E_0 \subset \Co \Omega$ and $E$ be an $s$-minimal set in~$\Omega$
with respect to $E_0$. Then
	\[ 
		\G(\chi_E, \Omega) \leq \G (u, \Omega)
		\]
		for any $u \in \W^s(\Omega)$ such that $u= \chi_{E_0}$ in $\Co \Omega$. 		
\end{theorem}

Now we turn our attention to the existence of minimisers. For this,
it is convenient to introduce the family of competitors for a given
datum on the external domain. That is,
given a function $\varphi \colon \Co \Omega \to \R$, we denote
	\begin{equation}\label{WsOM}
			\W^s_\varphi(\Omega) :=\big\{ u\in \W^s(\Omega) \; \big| \; u=\varphi \mbox{ in } \Co \Omega \big\}.
	\end{equation}
Moreover, for any $t>0$,
we set 
	\begin{equation}\label{Ot}
	 \Omega_t:= \big \{ y \in \Rn \; \big| \; \dist(y,\Omega)<t \big \}.
	 \end{equation}
With this notation,
in terms of existence theories, we utilise the level sets method to
prove existence of minimisers, according to Definition~\ref{defn11},
provided that the
datum is bounded in a suitable neighbourhood of the domain~$\Omega$, or, more generally,
under an integral control of the ``local tail''
of the datum. More precisely, we have the following:

\begin{theorem}\label{INTAL}
There exists $\Theta = \Theta(n,s)>1$ such that
the following statement holds true.
If~$\varphi \colon \Co \Omega \to \R$ is such that
	\begin{equation} \label{Thbe2w4}\int_\Omega\bigg[
\int_{\Omega_{\Theta \diam(\Omega) } \setminus \Omega} \frac{ |\varphi(y)|}{|x-y|^{n+s}} \, dy
		 \bigg]\,dx <+\infty,
	 \end{equation}
	then there exists an
$s$-minimal function~$u\in \W_\varphi^s(\Omega)$ in~$\Omega$. 
\end{theorem}

Theorem~\ref{INTAL} can be seen as a fractional counterpart of
a classical method in~\cite{SWZ}, which
constructs a function of least gradient
by means of level sets with minimal perimeter. 
See also~\cite{MR1246349}
for related classical results.

In our framework, a different type of existence theory, based on a ``global condition''
on the tail, will  be presented in Theorem~\ref{APP:TH:E}.
\medskip

To complete our analysis of $s$-minimal sets, we provide a simple and explicit
example of non-uniqueness: in particular, we show that when the external datum
is a sector in the plane, then there are at least two different $s$-minimal sets
according to Definition~\ref{PERMI}.

\begin{theorem}\label{NONUN}
Let $E_0  \subset \R^2$ be defined as
	\[
		E_0:= \big\{ (x,y)\in \Co B_1 \; \big| \; xy >0\big\}.
	\]
		Then there exist at least two different $s$-minimal sets in $B_1$
with respect  to $E_0$.
	\end{theorem}

In light of Theorem~\ref{TUTT}, Theorem~\ref{NONUN}
stresses the 
 importance of  the {\em lack of convexity} of the functional~$\G$, as it can indeed lead to {\em lack of uniqueness} of minimisers. 
This should be compared with the minimisation problem for the nonlocal fractional area type functionals in \cite{teolu}, for which the uniqueness is ensured by the strict convexity of those kinds of functionals.

\subsection{A Yin-Yang theorem}

Now we present a result that states, roughly speaking, that
if the external set is void (respectively, full) in a sufficiently large
neighborhood of the domain,
then the minimiser of the fractional perimeter is also void (respectively, full).

\begin{theorem}\label{thm13}
There exists $\Theta:= \Theta(n,s)>1$ such that the following
statement holds true.
\\
Let~$E_0\subset \Co \Omega$ and $E$ be an
$s$-minimal set in~$\Omega$ with respect to $E_0$.
\\
If $$E_0\cap \left(\Omega_{\Theta \diam(\Omega)} \setminus 
\Omega\right)=\emptyset,$$ then 
	\[
		E \cap \Omega= \emptyset.
	\] 
	Similarly, if  \begin{equation}\label{F678AG:BAL:1}
E_0\cap \left(\Omega_{\Theta \diam(\Omega)} \setminus \Omega\right)
=\Omega_{\Theta \diam(\Omega)} \setminus \Omega,\end{equation} then 
	\begin{equation}\label{F678AG:BAL:2}
		E \cap \Omega= \Omega.
	\end{equation}
\end{theorem}

\begin{figure}[h!]
\includegraphics[scale=0.3]{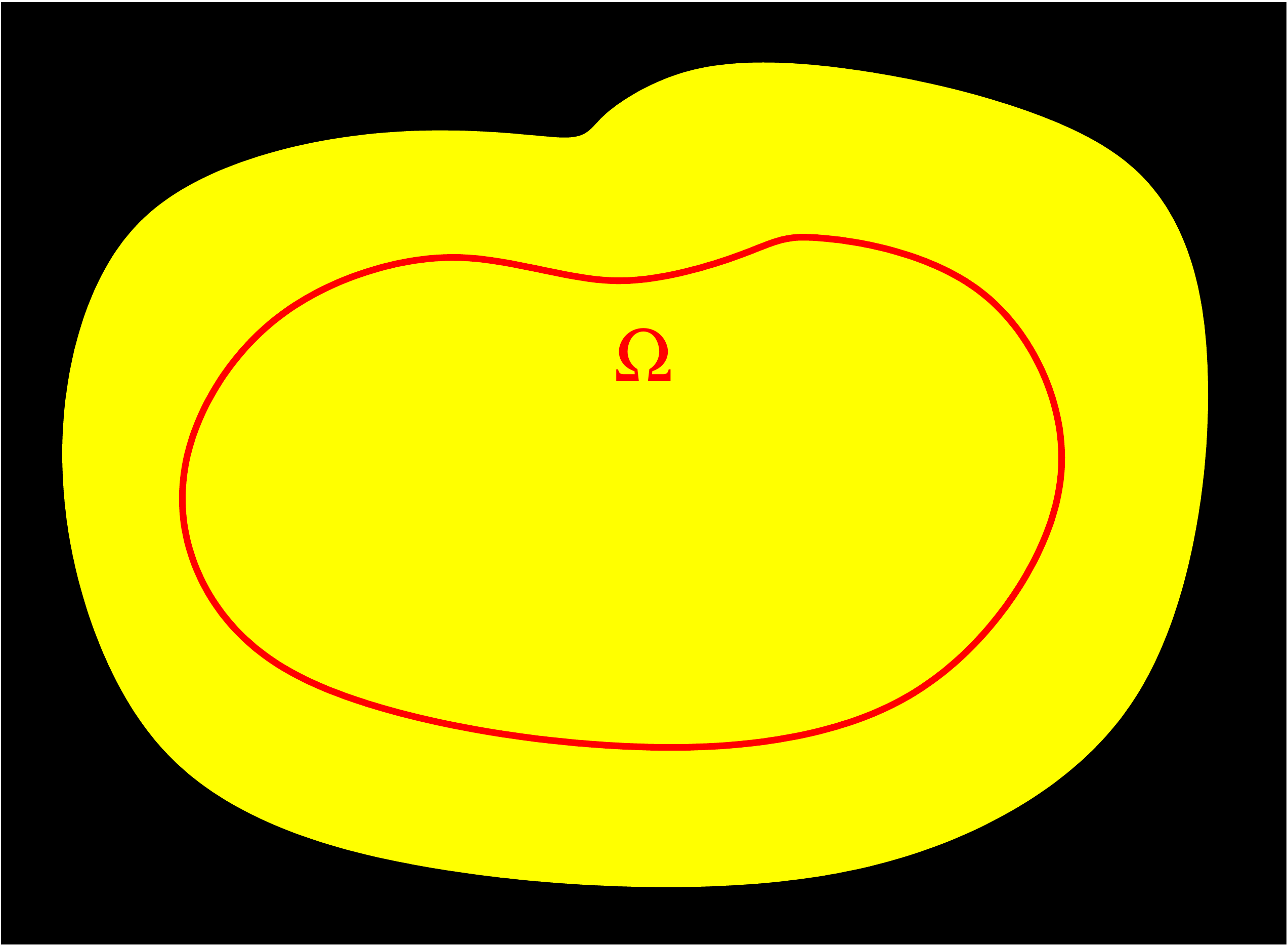}$\qquad$\includegraphics[scale=0.3]{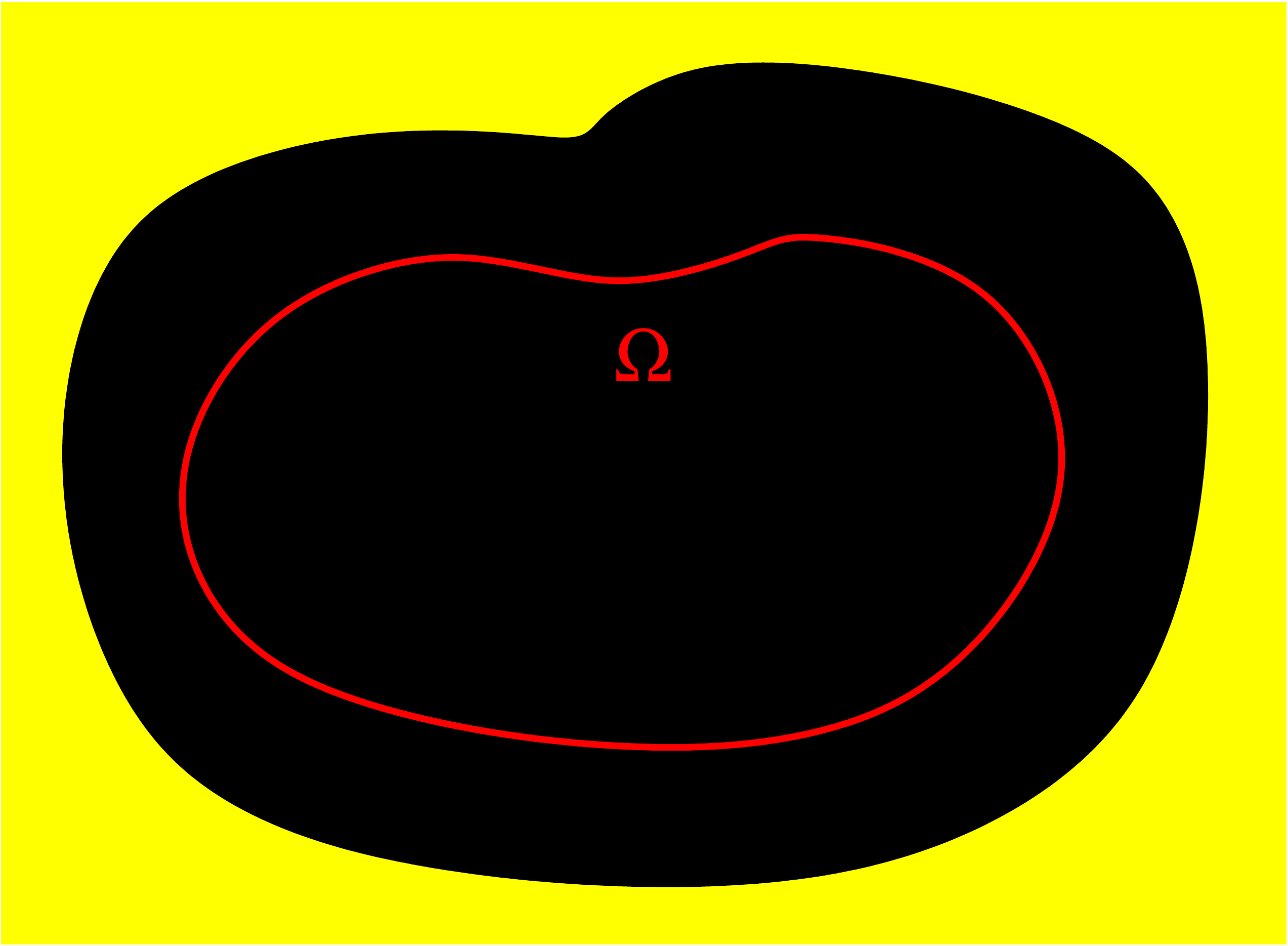}
\caption{The Yin and Yang result given in Theorem~\ref{thm13}.}\label{FIG1}
\end{figure}

In our framework, in spite of its clear and independent geometric interest,
Theorem~\ref{thm13} happens, on the one hand, to be a spin-off
of the techniques developed to prove Theorem~\ref{INTAL},
and, on the other hand, to play an important role in the proof of Theorem~\ref{INTAL}
by providing a useful pointwise and integral bound.

The reason for which we consider Theorem~\ref{thm13} as a Yin-Yang result
comes from its pictorial representation in Figure~\ref{FIG1},
which somehow resembles the
``Yin and Yang'', or ``Tai Chi'' Symbol, in which one
develops ``white'' parts inside the ``black'' ones,
in 
a blend of ``full'' and ``empty'' patterns.

We think that Theorem~\ref{thm13} gives a new and interesting insight
into nonlocal minimal surfaces, since it states that the behaviour of the datum at infinity
is not sufficient to produce nontrivial minimal sets inside a given domain.
Quite surprisingly, only data which are nontrivial near the boundary can give rise
to nontrivial minimisers: this phenomenon seems to us particularly relevant especially
when compared with the ``stickiness'' effects detected in~\cite{boundary, graph, bucluk, sticks, stickscrelle},
since in those circumstances the far-away data are able to produce very significant
effects. In this respect, the situation described in Theorem~\ref{thm13} can be
seen as a striking counterpart of the stickiness effect, since in the case treated here the far-away
interactions cannot play any significant role.

In particular, one can compare
Theorem~\ref{thm13} here with
Theorems~1.4 and~1.7 in~\cite{bucluk}. 
Roughly speaking, \cite[Theorems~1.4,~1.7]{bucluk} 
provide conditions under which,
if $s$ is sufficiently small, when
the external datum is empty at infinity,
then the $s$-minimal set is {\em empty} inside the domain. With respect to this,
Theorem~\ref{thm13} here says that, fixed~$s$, if
the external datum is full near the domain, and possibly empty at infinity,
then the $s$-minimal set is {\em full} inside the domain. That is,
a quantitative comparison between the ``smallness of~$s$'' and the ``width on
which the external datum is full'' play a crucial role in determining
the shape of the $s$-minimal set inside the domain, and {\em merely
qualitative arguments are not sufficient} to detect the different phenomena
of filling or emptying the domain.

As a matter of fact, Yin-Yang results for nonlocal minimal surfaces can only arise
from a fine balance between the mass of the prescribed set ``in the vicinity
of the domain'' and the one ``at infinity'', since
external data which are confined in a given neighborhood of the domain
end up producing void $s$-minimal sets in the domain for sufficiently small~$s$
(see~\cite[Theorems~1.4,~1.7]{bucluk}): in particular,
Theorem~\ref{thm13} here is optimal in the sense that if~\eqref{F678AG:BAL:1}
is replaced by
\begin{equation}\label{JAX:83rtu}
E_0=\Omega_{R} \setminus \Omega,\end{equation}
for a given~$R>0$, then, if~$s$ is sufficiently small,
\eqref{F678AG:BAL:2} does not hold true.
More precisely, let~$\{E_s\}_{s\in(0,1)}$ be any family of sets such that~$E_s$ is $s$-minimal in~$\Omega$ with respect to~$E_0$. Then, the density estimate given by~\cite[Theorem~1.5]{bucluk} ensures that there exists some small threshold index~${s}_\star\in(0,1)$ such that, if~$s\in(0,{s}_\star)$, then~$E_s\cap\Omega\not=\Omega$, and indeed, as observed in~\cite[Corollary~3.1]{bucluk}, in the limit~$s\searrow0$ the $s$-minimal sets become empty in the domain, i.e.
\[
\lim_{s\searrow0}|E_s\cap\Omega|=0.
\]

For the sake of completeness, we also point out that
Yin-Yang results for classical minimal surfaces are just trivially satisfied, in any case and
independently on any quantitative consideration on the ring around the domain
that is taken into account: indeed, if~\eqref{JAX:83rtu}
holds true for a given~$\delta>0$ and~$E$ is a minimiser for the classical
perimeter in the closure of~$\Omega$ with~$E=E_0$
in~$\Co\Omega$ (or, equivalently, with~$E=E_0$ along~$\partial\Omega$),
then obviously
$$ E \cap \Omega= \Omega,$$
since this configuration produces zero classical perimeter in the closure of~$\Omega$.
These observations also highlight the deep structural differences
of the nonlocal case with respect to the classical one, as far as Yin-Yang results are concerned.

\subsection{Open problems and possible future projects}
We think that the results
presented in this paper can open several fascinating
directions for future investigations. To start with, we plan to study
the Euler-Lagrange equation satisfied
by the $s$-minimal functions and to investigate its connection to the $(s,p)$-fractional
Laplacian, in the limit as $p\searrow1$, in the spirit of~\cite{MR2164415, MR3263922}.

Another interesting question is that of understanding the regularity of
$s$-minimal functions. On the one hand, $s$-minimal functions are in general not even continuous
(since characteristic functions of sets can be $s$-minimal, recall Theorem~\ref{TUTT}).
On the other hand, it is proved in~\cite{MR1246349}
that in the classical framework the minimiser
constructed via the level set method is continuous: this relies on
a strict comparison principle satisfied by classical minimal surfaces.
The investigation of the validity of general fractional
comparison principles is a compelling direction of research also
in itself. In any case, the nonlocal setting presents its own difficulties,
and comparison principles alone may not suffice to establish the
continuity of the $s$-minimal functions constructed by level sets methods
in Theorem~\ref{exlev}.

A very challenging, but extremely important, direction of research
is also related to the regularity of nonlocal minimal surfaces. Indeed,
in the classical case, one can show that the Simons' cone
is realised as a level set of a function with least gradient:
in this setting, results relating minimal functions to minimal sets are obtained and
exploited in~\cite{bomby}
in order to deduce the minimality of the Simons' cone, thus
providing an example of a singular cone minimising the classical perimeter.
In our framework, if one could manage to realise a singular cone
as a level set of an $s$-minimal function, then a direct application
of Theorem~\ref{STER} here would establish the $s$-minimality
of such a cone and thus provide an example of a singular nonlocal
minimal surface.

\subsection{Organisation of the paper}
The rest of this article is organised as follows. In Section~\ref{uno}
we analyse the structure of the
level sets of the $s$-minimal functions of Definition~\ref{defn11}
and relate them to the $s$-minimal sets of Definition~\ref{PERMI},
also proving Theorem~\ref{STER}.

Then, in Section~\ref{GSmshde23d}, we study the
characteristic functions of the $s$-minimal sets according to
Definition~\ref{PERMI} and we prove Theorem~\ref{TUTT}.

Section~\ref{GA:eucjex} deals with the
existence theory of the $s$-minimal functions of Definition~\ref{defn11} and with
the proof of Theorem~\ref{INTAL}.
For this, one also needs Theorem~\ref{thm13} in order to obtain
suitable integrable bounds, hence part of Section~\ref{GA:eucjex}
is also devoted to Yin-Yang results.

Then, Section~\ref{SJifngwhis}
presents an example of two $s$-minimal sets
sharing the same external datum, thus proving Theorem~\ref{NONUN}. 

The article ends with an appendix that collects some ancillary results.

\section{Level sets of $s$-minimal functions, $s$-minimal sets,
and proof of Theorem~\ref{STER}}\label{uno}
 
In this section we analyse the level sets of the $s$-minimal functions obtained in view of
Definition~\ref{defn11} and we relate them to the $s$-minimal sets
described in Definition~\ref{PERMI}.

To this end, we introduce some notation.
As customary, we will consider the ``positive'' and ``negative'' parts
of a function, defined by
 \bgs{
 	 u_+(x):= \max\{ u(x),0\}, \quad u_{-}(x):=\min\{ u(x),0\}.
  }
We notice that
 \eqlab{
 	\label{pm}
 	 u(x) = u_+(x) +u_-(x), \quad \mbox{ and } \quad |u(x)|= u_+(x)-u_-(x), \quad \mbox{ or }  \quad|u(x)|=|u_+(x)|+|u_-(x)|.
 	 }
Also, in the setting of~\eqref{PER}, it is convenient to write
       \bgs{
        \Per_s(E,\Omega)= \Per_s^{L} (E,\Omega) + \Per_s^{NL} (E,\Omega),} 
      where
    \begin{equation}\label{PLa}\begin{split}&
    	 	\Per_s^{L} (E,\Omega) := \frac12 \int_{\Omega}\int_{\Omega}
    	 	 \frac{|\chi_{E}(x) -\chi_E(y)|}{|x-y|^{n+s}} \, dx \, dy
    	 	 =\frac{1}{2}[\chi_E]_{W^{s,1}(\Omega)} \quad {\mbox{and }}
    	 	\\\qquad&\Per_s^{NL} (E, \Omega) := \int_{\Omega}\int_{\Co \Omega} \frac{|\chi_{E}(x) -\chi_E(y)|}{|x-y|^{n+s}} \, dx \, dy\end{split}
         \end{equation}
       represent respectively the   ``local'' and the ``nonlocal'' contributions to the fractional perimeter
(of course, even the ``local'' contribution is in fact of nonlocal type,
but the interactions are confined in the domain~$\Omega$). \medskip

Now, we relate the nonlocal perimeter to a renormalised version
of the functional in~\eqref{GAG}. 
To this aim, given $u \in \W^s(\Omega)$, we define the function
 	\sys[\tilde u :=]
 	{
 		&u & \mbox{ in  } \; & \Co \Omega,\\
 		&0& \mbox{ in  } \; & \Omega.
 	}
Notice that, clearly,
\begin{equation}\label{ALms3uencncdX}
{\mbox{if~$u=v$ in~$\Co \Omega$, then~$\tilde u=\tilde v$.}}
\end{equation}
Then, we define
	\begin{equation}\label{tG}
		\tilde \G (u,\Omega):=
		 \frac12 \iint_{Q(\Omega)} \big(|u(x) -u(y)|-|\tilde u (x) -\tilde u(y)|\big) 
		 \frac{dx\, dy}{|x-y|^{n+s}}  .
		\end{equation}
We observe that the setting of the functional $\tilde \G$ in~\eqref{tG}
is well posed, thanks to~\eqref{GHE}. 
The reader can also compare~\eqref{tG} with~\eqref{GAG}, to see
in which sense we consider~$\tilde \G$ a renormalised version of~$\G$: in our framework,
$\tilde \G$
presents the advantage of ``canceling'' the common tails of the integrands (which could
in principle be divergent when considered separately).
 \medskip

Interestingly, the minimisation of~$\tilde \G$ is directly related to
the notion introduced in Definition~\ref{defn11}, and also to the minimisation of~$\G$,
if the tail of the functional is finite.
We consider the ``global tail'' of the functional in~\eqref{GAG}, that is the contribution
coming from the interactions involving points outside the domain~$\Omega$.
To this end, we define
	\eqlab{ \label{coda}
		 \mathcal T_s(u,\Omega) := \int_{\Omega} \left( \int_{\Co \Omega} \frac{ |u(y)|}{|x-y|^{n+s}} dy \right) dx .
	}
	
\begin{lemma}\label{TGAet}
Let~$u\in \W^s(\Omega)$.
Then, $u$ is an $s$-minimal function according to Definition~\ref{defn11}
if and only if 
	\begin{equation}\label{HS:ksncb7rhhbvcsjjs}
	\tilde \G (u,\Omega) = \inf \{ \tilde \G (v,\Omega) \, | \, v\in \W^s(\Omega), u=v \mbox{  almost everywhere in  } \Co \Omega \}.
	\end{equation}
Moreover, if
\begin{equation}\label{HS:ksncb7rhhbvcsjjs2}
\mathcal T_s(u,\Omega) <+\infty,\end{equation}
then $u$ is an
$s$-minimal function
according to Definition~\ref{defn11} if and only if it is a minimiser of $\G$.	
\end{lemma}

\begin{proof} If~$u$, $v\in \W^s(\Omega)$ with~$u=v$ in~$\Co\Omega$,
we see that
\bgs{
	2\big(
\tilde \G (u,\Omega) - \tilde \G (v,\Omega)\big)=&\;
\iint_{Q(\Omega)} \big(|u(x) -u(y)|-|\tilde u (x) -\tilde u(y)|\big) 
		 \frac{dx\, dy}{|x-y|^{n+s}}\\
		 &\; -
\iint_{Q(\Omega)} \big(|v(x) -v(y)|-|\tilde v (x) -\tilde v(y)|\big) 
		 \frac{dx\, dy}{|x-y|^{n+s}}.
}
Since the last two integrands are summable, thanks to~\eqref{GHE} and recalling~\eqref{ALms3uencncdX},
we  obtain that
\bgs{
 &\; 2\big(
\tilde \G (u,\Omega) - \tilde \G (v,\Omega)\big)\\
=&\; 
\iint_{Q(\Omega)} \big(|u(x) -u(y)|-|\tilde u (x) -\tilde u(y)|-
|v(x) -v(y)|+|\tilde v (x) -\tilde v(y)|\big) 
		 \frac{dx\, dy}{|x-y|^{n+s}}\\
=&\;
\iint_{Q(\Omega)} \big(|u(x) -u(y)|-
|v(x) -v(y)| \big) 
		 \frac{dx\, dy}{|x-y|^{n+s}}.
	}
This yields that~$u$ is an $s$-minimal function according to Definition~\ref{defn11}
if and only if~\eqref{HS:ksncb7rhhbvcsjjs} is satisfied.

Let us now suppose that~\eqref{HS:ksncb7rhhbvcsjjs2} holds true.
Given the definition of $\tilde u$, we have that
		 		\eqlab{ \label{spliten}
		\tilde \G (u, \Omega) =		
		\frac12 \int_{\Omega} \int_{\Omega} \frac{|u(x)-u(y)|}{|x-y|^{n+s}}
		 \, dx \, dy 
		 + \int_{\Omega}  \left( \int_{\Co \Omega} \frac{|u(x)-u(y)| -|u(y)|}{|x-y|^{n+s}}  \, dy\right) dx
		.
		 }
		 Using the triangle inequality, we observe that
		 \bgs{
		 \iint_{Q(\Omega)} \frac{ |u(x)-u(y)|}{|x-y|^{n+s}} dx\, dy
		 \leq &\; [u]_{W^{s,1}(\Omega)} + 2 \mathcal T_s(u,\Omega) + 2 \int_\Omega dx\left(\int_{\Co \Omega} \frac{|u(x)|}{|x-y|^{n+s}} \, dy\right) 
		 \\
		 \leq &\; C\|u\|_{W^{s,1}(\Omega)} + 2 \mathcal T_s(u,\Omega)
		 ,}
		 by~\eqref{frch}.
		 Therefore we can write~\eqref{spliten} as
		 \begin{equation}\label{GAjgisq4}
		 	\tilde \G(u,\Omega) = \G(u,\Omega)-\mathcal T_s(u,\Omega),
		 \end{equation}
given that both terms are finite.	 
Therefore, for any $v\in \W^s(\Omega)$ such that $u=v $ in $\Co \Omega$ it holds that
$$ \tilde \G (u,\Omega)-\tilde \G (v,\Omega)=\G (u,\Omega)- \G (v,\Omega),$$
from which we obtain that $u$ is an
$s$-minimal function
according to Definition~\ref{defn11} if and only if it is a minimiser of $\G$.	
\end{proof}

Using the notation above,
we will use the following co-area formula for the functional $\tilde \G$
related to 
the fractional perimeter of the corresponding level sets.

\begin{proposition}\label{coarea}
Let $\Omega \subset \Rn$ be a bounded open set and let  $u\colon \Omega \to \R$. Then
	\eqlab{ \label{hhh}
	\frac12[ u]_{W^{s,1}(\Omega)} = \int_{-\infty}^\infty \Per_s^{L}(\{ u\geq t\}, \Omega) \, dt
	.}
In particular $u\in W^{s,1}(\Omega)$ if and only if the right hand side of equation~\eqref{hhh} is finite.\\
Moreover, if $\Omega$ has Lipschitz boundary, it holds that
	\eqlab{
	\label{jhgf}
	\tilde \G (u,\Omega) = \int_{-\infty}^{\infty} \Big(\Per_s( \{u \geq t\}, \Omega ) - \Per_s( \{\tilde u \geq t\}, \Omega ) \Big) \, dt
	}
	for any $u\in \W^s(\Omega)$.
\end{proposition}
\begin{proof}
	According to~\cite{DePhil,visintin} (see also~\cite{approxLuk,bucval}), by Fubini-Tonelli we have~\eqref{hhh}.
	We remark moreover that  if the seminorm~$[u]_{W^{s,1}(\Omega)}$ is finite, by~\cite[Lemma D.1.2]{tesilu}, we have that~$u\in L^1(\Omega)$, hence indeed $u\in W^{s,1}(\Omega)$.
This proves the second statement in
Proposition~\ref{coarea}, and we now focus on the proof of
formula~\eqref{jhgf}.
\\
For this, we notice that
	\sys[\{ \tilde u \geq t\}\cap \Omega=]
		{ 
		&\Omega  && {\mbox{if }}t\leq 0 ,\\
		&\emptyset&& {\mbox{if }}t>0
		,}
		which implies that $\Per_s^L (\{ \tilde u \geq t\},\Omega)=0$ for every $t\in \R$.\\
Thus, by~\eqref{hhh},
		\eqlab{\label{unloc}
		\frac12	\int_{\Omega} \int_{\Omega} \frac{|u(x)-u(y)|}{|x-y|^{n+s}} \, dx \, dy 
		=  
	\int_{-\infty}^\infty \Big( \Per_s^L( \{u\geq t\},\Omega)-\Per_s^L (\{ \tilde u \geq t\},\Omega)\Big) \, dt.
	}
	On the other hand, 
we observe that
\bgs{
	\int_0^\infty \chi_{ \{ u\geq t \}} (x) \, dt = u_{+}(x), \qquad{\mbox{and}}\qquad 	-\int_{-\infty} ^0 \chi_{ \{ u\leq t \}} (x) \, dt = u_{-}(x).
}
Using this and~\eqref{pm}, we find that
	\eqlab{ \label{288}
	\int_{-\infty}^\infty  |\chi_{ (-\infty,0] } (t) -\chi_{ \{u\geq t \}} (y)|\, dt
\,=\,& \int_{0}^\infty  \chi_{ \{u\geq t \}} (y)\, dt+
	\int_{-\infty}^0  \big(1 -\chi_{ \{u\geq t \}} (y)\big)\, dt
\\	 =\,&\int_{0}^\infty   \chi_{ \{u\geq t \}} (y)\, dt + \int_{-\infty}^0
\chi_{ \{u< t \}} (y)\, dt 
\\=\,& u_+(y)-u_-(y)\\=\,&|u(y)|.
	}
Furthermore, for any~$x, y\in\R^n$, we have that
\begin{equation}\label{SIM34}
|u(x)-u(y)|=\int_{-\infty}^{+\infty}|\chi_{ \{u\ge t\} } (x) -\chi_{ \{u\geq t \}} (y)|\, dt.
\end{equation}
To check this, we can suppose that~$u(x)\ge u(y)$. In this way, we have that
if~$t<u(y)$ then~$\chi_{ \{u\ge t\} } (x) =\chi_{ \{u\geq t \}} (y)=0$,
and if~$t>u(y)$ then~$\chi_{ \{u\ge t\} } (x) =\chi_{ \{u\geq t \}} (y)=1$.
This yields that
$$ \int_{-\infty}^{+\infty}|\chi_{ \{u\ge t\} } (x) -\chi_{ \{u\geq t \}} (y)|\, dt=
\int_{u(y)}^{u(x)}|\chi_{ \{u\ge t\} } (x) -\chi_{ \{u\geq t \}} (y)|\, dt
=\int_{u(y)}^{u(x)}\, dt=u(x)-u(y),$$
proving~\eqref{SIM34}.
\\
Then, making use of~\eqref{288}
and~\eqref{SIM34},
we conclude that
	\bgs{
	&\;	\int_{\Omega}  \left( \int_{\Co \Omega} \frac{|u(x)-u(y)| -|u(y)|}{|x-y|^{n+s}}  \, dy\right) dx
		\\
		=&\;
		\int_{\Omega}\left( \int_{\Co \Omega}  
		\left( \int_{-\infty}^\infty |\chi_{ \{u\geq t \} } (x) -\chi_{ \{u\geq t \}} (y)| 
			- |\chi_{ (-\infty,0] } (t) -\chi_{ \{u\geq t \}} (y)|\, dt 				
		\right)
			\frac{ dy}{|x-y|^{n+s}} \right) \, dx.
	}
In order to exchange the integrals, we need to check that the integrand is summable. For this, we remark that 
\bgs{
&	\int_{\Omega} \int_{\Co \Omega} \int_{-\infty}^\infty \big| |\chi_{ \{u\geq t \} } (x) -\chi_{ \{u\geq t \}} (y)| 
	- |\chi_{ (-\infty,0]} (t) -\chi_{ \{u\geq t \}} (y)| \big|   	\frac{dt \, dy \, dx}{|x-y|^{n+s}},
	\\
	\leq &\; 	\int_{\Omega} \int_{\Co \Omega} \int_{-\infty}^\infty \big| \chi_{ \{u\geq t \} } (x) - \chi_{ (-\infty,0] } (t) \big|   	\frac{dt \, dy \, dx}{|x-y|^{n+s}}
	\\
	= &\; 	\int_{\Omega} \left(\int_{\Co \Omega}  \frac{ |u(x) |} {|x-y|^{n+s}}  \, dy\right) dx\leq C \|u\|_{W^{s,1}(\Omega)},
	}
	thanks to~\eqref{288} and to~\eqref{frch}. \\
Thus, 
since,  given $t\in \R$,
\begin{eqnarray*}&&
		\chi_{ \{ \tilde u \geq t\} } (x) =\chi_{(-\infty,0]}(t)\qquad{\mbox{for any $x\in \Omega$,}} \quad {\mbox{and }}\\
&&
\chi_{ \{ \tilde u \geq t\} } (x) =\chi_{\{u\ge t\}}(x)\qquad{\mbox{for any $x\in \Co\Omega$,}}
	\end{eqnarray*}
we obtain 
	\bgs{
	&\;	\int_{\Omega}  \left( \int_{\Co \Omega} \frac{|u(x)-u(y)| -|u(y)|}{|x-y|^{n+s}}  \, dy\right) dx
		\\
		=&\;
	\int_{-\infty}^\infty \left(	\int_{\Omega} \left(\int_{\Co \Omega}  
		 |\chi_{ \{u\geq t \} } (x) -\chi_{ \{u\geq t \}} (y)| 
			- |\chi_{ (-\infty,0] } (t) -\chi_{ \{u\geq t \}} (y)| 				
		\frac{ dy}{|x-y|^{n+s}}\right) \, dx \right) \, dt
		\\=&\;
	\int_{-\infty}^\infty \left(	\int_{\Omega} \left(\int_{\Co \Omega}  
		 |\chi_{ \{u\geq t \} } (x) -\chi_{ \{u\geq t \}} (y)| 
			- |\chi_{ \{\tilde u\ge t\} } (x) -\chi_{ \{\tilde u\geq t \}} (y)| 				
		\frac{ dy}{|x-y|^{n+s}}\right) \, dx \right) \, dt
		\\
		=&\; \int_{-\infty}^\infty\left( \Per^{NL}_s(\{u\geq t \},\Omega) - \Per^{NL}_s(\{\tilde u\geq t \},\Omega)\right) \, dt.
	}
This and~\eqref{unloc},  recalling also~\eqref{spliten},
give the desired result in~\eqref{jhgf}.
\end{proof}

\begin{remark} \label{coareag} Let $u,v\in \W^s(\Omega)$ such that $u=v$ almost everywhere in $\Co \Omega$. Then, by~\eqref{ALms3uencncdX} and~\eqref{jhgf} we have 
\bgs{
 	&\iint_{Q(\Omega)} |u(x) -u(y)|-|v(x)  - v(y)|\, \frac{dx\, dy}{|x-y|^{n+s}} 
 	\\
 	&=\int_{-\infty}^{\infty} \big(\Per_s( \{u \geq t\}, \Omega ) - \Per_s( \{v \geq t\}, \Omega ) \big) \, dt.
 	}
Moreover, if~$\mathcal T_s(u,\Omega)<+\infty$, one can use~\eqref{jhgf}
and~\eqref{GAjgisq4} to see that
\begin{equation}\label{CoA con G}
\begin{split}
\G(u,\Omega)\,&=
\tilde \G(u,\Omega) +\mathcal T_s(u,\Omega)\\
&=\int_{-\infty}^{\infty} \Big(\Per_s( \{u \geq t\}, \Omega ) - \Per_s( \{\tilde u \geq t\}, \Omega ) \Big) \, dt
+\mathcal T_s(u,\Omega)\\
&=\int_{-\infty}^{\infty} \Big(\Per_s( \{u \geq t\}, \Omega ) - \Per_s^{NL}( \{\tilde u \geq t\}, \Omega ) \Big) \, dt
+\mathcal T_s(u,\Omega)\\
&=\int_{-\infty}^{\infty} \Per_s( \{u \geq t\}, \Omega )\,dt.
\end{split}\end{equation}
\end{remark}

Now we point out that the functional $\tilde \G$ can be ``nicely
split'' between the positive and the negative parts. Interestingly,
this useful fact is not a direct consequence of~\eqref{tG}, but it rather relies
on the co-area formula stated in Proposition~\ref{coarea}.

\begin{lemma}\label{LAxS:2.5}
For every $u\in \W^s(\Omega)$, it holds that
\[ \tilde \G (u,\Omega)= \tilde \G (u_+,\Omega) +\tilde \G (u_-,\Omega).\]
\end{lemma}
\begin{proof}
We notice that for any $t\geq 0$
		\[ 	 \{u \geq t\}= \{u_+ \geq t\}\qquad{\mbox{and}}  \qquad \{u _-\geq  t\}= \emptyset ,\]
		while for $t<0$
			\[ 	 \{u \geq t\}= \{u_- \geq t\}\qquad{\mbox{and}} \qquad \{u _+\geq  t\}=\Rn.\]
Therefore,
we have that
\bgs{ 
		&\int_{-\infty}^{\infty} \big(\Per_s( \{u \geq t\}, \Omega ) - 	\Per_s( \{\tilde u \geq t\}, \Omega ) \big) \, dt
		\\
		= &\;
		\int_{0}^{\infty} \big(\Per_s( \{u \geq t\}, \Omega ) - \Per_s( \{\tilde u \geq t\}, \Omega ) \big) \, dt
		\\
		&\; + \int_{-\infty}^{0} \big(\Per_s( \{u \geq t\}, \Omega ) - \Per_s( \{\tilde u \geq t\}, \Omega ) \big) \, dt
		\\
			= &\;
		\int_{0}^{\infty} \big(\Per_s( \{u_+ \geq t\}, \Omega ) - \Per_s( \{\tilde u_+ \geq t\}, \Omega ) \big) \, dt
		\\
		&\; + \int_{-\infty}^{0} \big(\Per_s( \{u_- \geq t\}, \Omega ) - \Per_s( \{\tilde u_- \geq t\}, \Omega ) \big) \, dt\\
		= &\;
		\int_{-\infty}^{\infty} \big(\Per_s( \{u_+ \geq t\}, \Omega ) - \Per_s( \{\tilde u_+ \geq t\}, \Omega ) \big) \, dt
		\\
		&\; + \int_{-\infty}^{\infty} \big(\Per_s( \{u_- \geq t\}, \Omega ) - \Per_s( \{\tilde u_- \geq t\}, \Omega ) \big) \, dt
		.}
		Hence, we use Proposition~\ref{coarea} to obtain the desired conclusion.
\end{proof}

A useful consequence of Lemma~\ref{LAxS:2.5} is that, given an
$s$-minimal function,
its positive and negative parts are $s$-minimal functions as well.

\begin{lemma} \label{minpm}
If $u\in \W^s(\Omega)$ is an $s$-minimal function, then also $u_+, u_-$ are
$s$-minimal functions.
\end{lemma}
\begin{proof}
Let $\psi \in \W^s(\Omega)$ such that $\psi=0$ in $\Co \Omega$. Since $u$ is an
$s$-minimal function, we have that
	\[
	 \tilde \G(u,\Omega) \leq \tilde \G (u+\psi,\Omega).
	 \]
We claim that 
\eqlab{
	\label{gg} &
	\tilde\G (u+\psi,\Omega) \leq \tilde \G (u_++\psi,\Omega) + \tilde \G (u_-,\Omega)\quad {\mbox{and }}\\
&\tilde\G (u+\psi,\Omega) \leq \tilde \G (u_+,\Omega) + \tilde \G (u_- +\psi,\Omega).
	}
Once this is done, using the first inequality in~\eqref{gg}
and Lemma~\ref{LAxS:2.5}, we obtain that
	\bgs{ 
	\tilde \G (u_+,\Omega)+\tilde \G (u_-,\Omega) = \tilde \G(u,\Omega) \leq \tilde \G (u+\psi,\Omega) \leq \tilde \G (u_++\psi,\Omega) + \tilde \G (u_-,\Omega),
	}
	so $u_+$ is an $s$-minimal function. 
Similarly, one can prove that~$u_-$
is an $s$-minimal function by using
the second inequality in~\eqref{gg}
and Lemma~\ref{LAxS:2.5}.
\\
Hence,  to complete the proof
of Lemma~\ref{minpm},
it remains to establish~\eqref{gg}. For this, we focus on the proof
of the first inequality, since the second one is similar.
	To this end, we split $\tilde \G$ into the two contributions, as in~\eqref{spliten}. Using the triangle inequality, we have that
	\begin{equation}\label{tgbs834t48rug}\begin{split}
	\int_\Omega \int_\Omega \frac{|(u+\psi)(x)-(u+\psi)(y)|}{|x-y|^{n+s}} \, dx \, dy 	
		\leq &\; \int_\Omega \int_\Omega \frac{|(u_++\psi)(x)-(u_++\psi)(y)|}{|x-y|^{n+s}} \, dx \, dy
	\\
	&\;  +\int_\Omega \int_\Omega \frac{|u_-(x)-u_-(y)|}{|x-y|^{n+s}} \, dx \, dy .
\end{split}\end{equation}
	On the other hand, by~\eqref{pm} and the fact that $\psi=0$ in $\Co \Omega$,
we see that
	\begin{equation}\label{tgbs834t48rug2}\begin{split}
	&\int_{\Omega}  \left( \int_{\Co \Omega} \frac{|(u+\psi)(x)-(u+\psi)(y)| -|(u+\psi)(y)|}{|x-y|^{n+s}}  \, dy\right) dx
	\\
	\leq &\;
	\int_{\Omega}  \left( \int_{\Co \Omega} \frac{|(u_++\psi)(x)-(u_++\psi)(y)| + |u_-(x)-u_-(y)|  -|u_+(y)|-|u_-(y)|}{|x-y|^{n+s}}  \, dy\right) dx.
\end{split}\end{equation}
As a consequence of~\eqref{tgbs834t48rug} and~\eqref{tgbs834t48rug2},
and recalling~\eqref{spliten}, we obtain the first inequality in~\eqref{gg}.
The proof of Lemma~\ref{minpm} is thereby complete.
\end{proof}

In our framework, a useful byproduct of Lemma~\ref{minpm} consists in the possibility
of ``cutting'' and ``rescaling'' an $s$-minimal function near to a given value
in order to emphasise the role played by its level sets. The precise result
that we need is the following:

\begin{lemma}\label{lemma19}
Let $u\in \W^s(\Omega)$ be an
$s$-minimal function in~$\Omega$. Then
\begin{equation} \label{phieps}
	 \varphi_{\lambda,\eps} := \frac{ \min\left\{ \eps, \max\left\{ u-\lambda+ \sqrt{\eps},0\right\} \right\}}{\eps}
\end{equation}
	is an $s$-minimal function in $\Omega$, for any $\lambda\in \R$ and~$\eps>0$.
\end{lemma}

\begin{proof}
 Fixing any $\lambda\in \R$ and $\eps>0$,  we observe that
\sys[ \varphi_{\lambda,\eps} =]{
	&1 & {\mbox{ if }}  & u\in[\lambda-\sqrt\eps+\eps,\,+\infty),
	\\ &\displaystyle\frac{u-\lambda+\sqrt\eps}{\eps}
	& {\mbox{ if }} &u\in(\lambda-\sqrt\eps,\,\lambda-\sqrt\eps+\eps),
		\\ &0 & {\mbox{ if }} &u\in(-\infty,\,\lambda-\sqrt\eps]
	}
and we notice that
  \begin{equation}\label{67t2uwsisir} 0\leq \varphi_{\lambda,\eps} \leq 1.\end{equation}
  First of all, we prove that any translation of an $s$-minimal function
is still an $s$-minimal function.
For this, let $v_\lambda \in \W^s(\Omega)$ be such
that $v_\lambda (x)= u(x)-\lambda $ for almost every $x \in \Co \Omega$,
and let $v(x) := v_\lambda (x) +\lambda$ (hence, $u=v$ almost everywhere in $\Co \Omega$). Then, since $u$ is an
$s$-minimal function, we have that  
  \bgs{
	&\iint_{Q(\Omega)} \big(|(u-\lambda)(x) -(u-\lambda)(y)|-|v_\lambda(x)  - v_\lambda(y)|\big) 
	\frac{dx\, dy}{|x-y|^{n+s}} 
	\\
	=&\;
		\iint_{Q(\Omega)} \big(|u(x) -u(y)|-|v(x)  - v(y)|\big) 
	\frac{dx\, dy}{|x-y|^{n+s}}  \leq 0,
  }
thus proving that~$u-\lambda$ is an $s$-minimal function in~$\Omega$.\\
  In a similar way, one proves that if $u$ is an
$s$-minimal function, then $M u$ is an
$s$-minimal function for any $M\in \R.$ 
  We notice that
	\[ 
	 \varphi_{\lambda,\eps}(x) = \frac{\left[\left(u-\lambda+  \sqrt{\eps}\right)_+-\eps\right]_- +\eps}{\eps}
	 ,\]
	 so by Lemma~\ref{minpm} and the previous considerations, we obtain the desired result.
\end{proof}

In our setting, the crucial geometric property of
the cutoff~$\varphi_{\eps, \lambda}$ defined in~\eqref{phieps}
is that of approximating the level sets of an
$s$-minimal function~$u$
(this geometric property will then be complemented with a ``variational
stability'' in the forthcoming Proposition~\ref{prop110}).
For the sake of clarity, we explicitly state the cutoff approximation
in the next result.

\begin{lemma}\label{lemma112}
Let $u\in \W^s(\Omega)$ be an
$s$-minimal function in~$\Omega$. We consider $\eps>0$, $\lambda \in \R$ and $\varphi_{\eps, \lambda}$ as defined in~\eqref{phieps}. Then 
	\[
	 \varphi_{\lambda,\eps} \longrightarrow  \chi_{\{u \geq \lambda\}}  
		 \qquad \mbox{ for } {\eps \to 0,} 
		 \qquad \mbox{ almost everywhere in} \quad \Rn.
	\]  
\end{lemma}
\begin{proof}
Notice that when $\chi_{\{u \geq\lambda\}}(x)=0$, then $\varphi_{\lambda, \eps} (x) =0$
for $\eps$ small enough. On the other hand, when $\chi_{\{u\geq \lambda\}}(x)=1$, then
	\[ 
		\varphi_{\lambda, \eps}(x) 
		= \min\left\{ 1, \frac{u(x)-\lambda + 	\sqrt{\eps} }{\eps} \right\} 
		\geq \min \left\{ 1,  \eps^{-\frac12} \right\}  =1,
	\]
	 for $\eps$ small enough. The conclusion follows from these observations.
\end{proof}

Now we focus our attention on the relation between the
properties of a given function and those of its level sets. In particular,
we show in the next lemma that if a function belongs to~$\W^s(\Omega)$, then almost every level set
has finite fractional perimeter.

\begin{lemma}
If $u\in \W^s(\Omega)$, then for almost any $\lambda \in \R$ 
\eqlab{
	\label{fin}
	 \Per_s( \{ u \geq \lambda\}, \Omega ) < +\infty.
	 }
\end{lemma}

\begin{proof}
	If $u\in \W^s(\Omega)$, then, by~\eqref{hhh},
	\[\int_{-\infty}^\infty \Per^L_s ( \{u \geq \lambda\} , \Omega) \, d \lambda   =  \frac12 \int_{\Omega} \int_{\Omega} \frac{|u(x)-u(y)|}{|x-y|^{n+s} } \, dx \, dy < +\infty,\]
	hence 
	\[ \Per^L_s ( \{u \geq \lambda\} , \Omega) <+\infty\]
	for almost any $\lambda \in \R$.
	Moreover
	\[ \Per_s^{NL}  ( \{u \geq \lambda\} , \Omega)  \leq 2 \Per_s(\Omega, \R^n) <+\infty,\]
	for every $\lambda \in \R$. The conclusion is established.
\end{proof}

Now we observe that the $s$-minimisation property for functions
is stable under convergence,
provided that the difference of the ``global tails'' given in~\eqref{coda} is also infinitesimal. The precise result
is the content of the next proposition.

\begin{proposition}\label{prop110}
Let $u_k\in \W^s(\Omega)$ be $s$-minimal functions such that
\begin{equation}\label{IM0}
		\sup_{k\in\N}
		\mathcal T_s (u_k,\Omega) <+\infty,
\end{equation}
and let $u\colon \Rn \to \R$ be such that
	\eqlab{\label{eq11} 
	u_k  \longrightarrow u \;  \mbox{ in }\;  L^1(\Omega), \qquad \mbox{ and } \qquad \mathcal T_s(u_k-u, \Omega) \to 0 \; \mbox{ for } \; k \to+\infty.
	}
Then
\begin{equation}\label{IM1}
u_k \longrightarrow u\; \mbox { for }\;  k \to+\infty \quad  \mbox{ in } \; L^1_{{\loc}}(\Rn)
\end{equation}
and
$u$ is an $s$-minimal function.
\end{proposition}

\begin{proof} First, we claim that
$u\in \W^s(\Omega).$
To this end, we observe that,
by Fatou's Lemma,
\begin{equation}\label{IM8}
 		[u]_{W^{s,1}(\Omega)} \leq \liminf_{k \to+\infty} [u_k]_{W^{s,1}(\Omega)}.
\end{equation}
Moreover, since $u_k$ is an
$s$-minimal function with bounded ``global tail''
as stated in~\eqref{IM0}, from Remark~\ref{TGAet}
we have that $u_k$ is a minimiser for $\G$. Hence, being $\tilde u_k$ a competitor for $u_k$, we have that
\bgs{
	\;\frac12[u_k]_{W^{s,1}(\Omega)} \leq \G(u_k,\Omega)\leq \G(\tilde u_k,\Omega) =\mathcal T_s(u_k,\Omega),
	}
	by definition of $\tilde u_k$. 
		Using this,~\eqref{IM0} and~\eqref{IM8}, we conclude that
$$ [u]_{W^{s,1}(\Omega)} \leq \sup_{k\in\N} [u_k]_{W^{s,1}(\Omega)}\le
2
\sup_{k\in\N}\mathcal T_s (u_k,\Omega)<+\infty,$$
thus proving that $u\in \W^s(\Omega)$. 
\\
Now we prove~\eqref{IM1}. For this, given any compact set  $K\subset \Rn$,
we have that
	\bgs{
		 \mathcal T_s (u_k-u,\Omega) \geq 	
	\int_{\Omega}\left(\int_{K\setminus \Omega} \frac{|u_k(y)-u(y)|}{|x-y|^{n+s}} \, dy \right) dx 
	\geq C |\Omega| \int_{K \setminus \Omega}|u_k(y)-u(y)| \, dy,
	}
	where $C>0$ depends on $\Omega$ and $K$.
	This, together with~\eqref{eq11}, gives the assertion in~\eqref{IM1}.
\\
To complete the proof of Proposition~\ref{prop110},
it remains to establish
that $u$ is an $s$-minimal function.
To this end, we observe that~$\mathcal T_s(u,\Omega)<+\infty$, 
thanks to~\eqref{IM0} and~\eqref{eq11}.
Hence,
recalling that $u\in \W^s(\Omega)$,
by Lemma~\ref{TGAet},
we know that $u$ is an $s$-minimal function if and only if it minimises $\G$. 
Consequently, we have that
the proof of Proposition~\ref{prop110} is complete
if we show that
\begin{equation}\label{TGV}
\G(u,\Omega) \le \G(v,\Omega)
\end{equation}
for any~$v\in \W^s(\Omega)$ such that $v=u$ in $\Co \Omega$.
To prove~\eqref{TGV}, we let
	\sys[v_k:=]{ &v & \mbox{ in } &\Omega,\\
					&u_k & \mbox{ in } &\Co \Omega.
					}
Then, the fact that $u_k$ is an
$s$-minimal function, together with Lemma~\ref{TGAet},
yields that
$$ \G(u_k,\Omega) \le \G(v_k,\Omega).$$
As a result, by Fatou's Lemma,
we have that
\begin{equation}\begin{split} \label{BNs83}
		\G(u,\Omega) -\G(v,\Omega) \leq &\; 
			\liminf_{k \to+\infty}  \G (u_k,\Omega)  -\G(v,\Omega) 
			\\
			\leq&\; \liminf_{k \to+\infty}  \big(
\G (u_k,\Omega) - \G(v_k,\Omega) \big)
			+ \limsup_{k \to+\infty}\G(v_k,\Omega) - \G(v,\Omega) 
			\\
			\leq &\; \limsup_{k \to+\infty}\big( \G(v_k,\Omega) - \G(v,\Omega) \big).
\end{split}\end{equation}	
	Now
	\bgs{
		\left| \G(v_k,\Omega) -\G(v,\Omega)\right|\leq &\;
			\int_{\Omega} \left( \int_{\Co \Omega} \frac{\big||v(x)-u_k(y)| - |v(x)-u(y)|\big|}{|x-y|^{n+s}} dy \right) dx
		\\
		\leq &\; \int_{\Omega} \left( \int_{\Co \Omega} \frac{|u(y)-u_k(y)| }{|x-y|^{n+s}} dy \right) dx\\
=&\mathcal T_s(u_k-u, \Omega).
		}			
Therefore, in light of~\eqref{eq11},
		\[ \lim_{k \to+\infty}\big( \G(v_k,\Omega) -\G(v,\Omega) \big)=0.\] 
This and~\eqref{BNs83} yield~\eqref{TGV}, as desired.
\end{proof}

In the next lemma, we give a sufficient condition for the convergence of the ``global tails''. 
\begin{lemma}\label{lemma111}
Let $u_k,u$ be such that
	\begin{equation}\label{con M}
		\|u_k\|_{L^\infty(\Rn)} \leq M
	\end{equation}
	for some $M>0$ and
\begin{equation}\label{con M2}
	 	u_k \longrightarrow u \qquad \mbox{  in }\;  L^1_{\loc}(\Rn) \qquad \mbox{ as }
k\to+\infty.
\end{equation}
	 Then 
\begin{equation*}\label{con M3} 
	 \mathcal T_s(u_k-u, \Omega) \to 0 \quad \mbox{ for } \quad  k \to+\infty.
\end{equation*}
\end{lemma}

\begin{proof} In this argument,
we follow the proof in~\cite[Theorem 3.3]{nms}, see also~\cite[Theorem 1.11]{approxLuk}.
We observe that, recalling~\eqref{coda} and~\eqref{con M},
	\bgs{
		\mathcal T_s(u_k,\Omega) 
		\leq M \, \Per_s(\Omega, \Rn)<+\infty.
	 } 
In the same way, using~\eqref{con M} and~\eqref{con M2}, we obtain that~$\mathcal T_s(u,\Omega)<+\infty$. 
\\
Now, recalling the notation in~\eqref{Ot},
we have that, if $\Omega$ is
a  bounded  open set with Lipschitz boundary, then
there exists $r_o>0$ such that, for all $t<r_o$, the sets $\Omega_t$ are bounded open sets  with Lipschitz boundary as well, and furthermore
	 \begin{equation}\label{effin}
	 	\sup_{t \in (0,r_o)} \mathcal H^{n-1} (\partial \Omega_t) <+\infty.
	 \end{equation}
	 We consider any $R>r_o$. Then
	 \begin{equation}\begin{split}\label{665544}
	 	\mathcal T_s(u_k-u,\Omega)
	 	= &\; \int_{\Omega} \left( \int_{\Omega_{r_o} \setminus \Omega} \frac{|u_k(y)-u(y)|}{|x-y|^{n+s}} dy\right) dx 
	 	+\int_{\Omega} \left( \int_{ \Omega_R \setminus \Omega_{r_o} } \frac{|u_k(y)-u(y)|}{|x-y|^{n+s}} dy\right) dx
	 	\\
	 	&\; +\int_{\Omega} \left(\int_{ \Co  \Omega_{R} } \frac{|u_k(y)-u(y)|}{|x-y|^{n+s}} dy\right) dx
	 	\\=:&\; I^k_{r_o} + I^k_{R,r_o} + I^k_{R}. 
	  \end{split}\end{equation}
	  We now estimate separately the three terms in~\eqref{665544}.
For this, we first observe that
	  \bgs{
	  	& I^k_{R}\leq \int_{\Omega} \left(\int_{ \Co B_R(x)} \frac{|u_k(y)-u(y)|}{|x-y|^{n+s}} dy\right) dx \leq C_{n,s} M |\Omega| R^{-s} \qquad \mbox{ and}
	  	\\
	  	& I^k_{R,r_o} \leq r_o^{-n-s} |\Omega| \int_{\Omega_R \setminus \Omega_{r_o}} |u_k(y)-u(y)|\, dy,
	  }
being~$C_{n,s}$ a positive constant (possibly varying from line to line).
	  Therefore,
$$ \lim_{R \to+\infty}\lim_{k \to+\infty}I^k_{R}\le
\lim_{R \to+\infty}\lim_{k \to+\infty}C_{n,s} M |\Omega| R^{-s}=0,$$
and, in addition, recalling~\eqref{con M2},
$$ \lim_{k \to+\infty}
I^k_{R,r_o} \leq r_o^{-n-s} |\Omega| \,
\lim_{k \to+\infty}\int_{\Omega_R \setminus \Omega_{r_o}} |u_k(y)-u(y)|\, dy=0.$$
As a result,
	  \begin{equation}\label{32023033}
	  	\lim_{R \to+\infty} \left[\lim_{k \to+\infty} \left(I^k_{R,r_o} + I^k_{R}\right) \right]= 0.
	  \end{equation}
	  Moreover, we see that
	  \begin{equation*}\label{SHn9345655}\begin{split}
	  	I^k_{r_o}= &\;
	  	\int_{\Omega_{r_o} \setminus \Omega} |u_k(y)-u(y)| 
	  	\left( \int_{\Omega} \frac{dx}{|x-y|^{n+s}} \right) dy
	  	\\
	  	\leq &\; \int_0^{r_o} 
	  	\left[\int_{\partial \Omega_t} |u_k(y)-u(y)| 
	  	\left( \int_{\Co B_t(y)} \frac{dx}{|x-y|^{n+s}} \right) d\mathcal H^{n-1}(y) \right] dt
	  	\\
	  	\leq &\;C_{n,s}  \int_0^{r_o} 
	  	\left(\int_{\partial \Omega_t} |u_k(y)-u(y)| 
	  	 d\mathcal H^{n-1}(y) \right) \frac{dt}{t^s}. 
	  \end{split}\end{equation*}
Recalling~\eqref{con M} and~\eqref{effin},
we also remark that
	  \[ \frac{1}{t^s}\,
	  \int_{\partial \Omega_t} |u_k(y)-u(y)| 
	  	 d\mathcal H^{n-1}(y) 
	  	 \leq 
	  	 \frac{2M}{t^s}\,\sup_{t\in (0,r_o)} \mathcal H^{n-1}(\partial \Omega_t) ,
	   \]
and the latter function is integrable when~$t\in(0,r_o)$.
\\
Consequently, exploiting~\eqref{eq11}
and the Dominated Convergence Theorem, we get that
	 \[
	  	 \lim_{k \to+\infty} \int_0^{r_o} 
	  	\left(\int_{\partial \Omega_t} |u_k(y)-u(y)| 
	  	 d\mathcal H^{n-1}(y) \right) \frac{dt}{t^s}  =0.
	  \] 
This
implies that
	\[
	 	 \lim_{k\to+\infty}  I^k_{r_o} =0.
	 \] 
{F}rom this, \eqref{665544} and~\eqref{32023033},
we obtain the claim, and conclude the proof of the Proposition.
\end{proof}

With this preliminary work, we are
ready to show that level sets
of $s$-minimal functions are $s$-minimal sets. The precise statement
goes as follows:

\begin{theorem} \label{thm211}
If $u\in \W^s(\Omega)$ is an
$s$-minimal function in~$\Omega$, then, for all $\lambda \in \R$,
\begin{equation}\label{SJkmd-1}
{\mbox{$\chi_{\{ u \geq \lambda \}}$ is an $s$-minimal function.}} \end{equation}
In addition,  for every $\lambda \in \R$,
\begin{equation}\label{SJkmd-2}
{\mbox{the sets $\{  u \geq \lambda \}$ are $s$-minimal in $\Omega$.}}\end{equation}
\end{theorem}

\begin{proof}
Let $\lambda\in \R$, $\eps>0$ and $ \varphi_{\lambda, \eps}$ be given in~\eqref{phieps}.
We consider an infinitesimal sequence~$\eps\to0$,
and we observe that, thanks to Lemma~\ref{lemma19},
\begin{equation*}\label{MEN1}
{\mbox{$\varphi_{\lambda,\eps}$ is an $s$-minimal function.}}
\end{equation*}
Moreover, in view of~\eqref{67t2uwsisir},
\begin{equation*}\label{MEN2}
\sup_{\eps\in(0,1)}\mathcal T_s (\varphi_{\lambda,\eps},\Omega) \le\Per_s(\Omega,\R^n)<+\infty.
\end{equation*}
In addition, by
Lemma~\ref{lemma112} and~\eqref{67t2uwsisir}, we have that
\begin{equation*}\label{MEN3}
{\mbox{$\varphi_{\lambda,\eps}\to \chi_{\{u\geq \lambda\}}$ in~$L^1(\Omega)$
\; \; as\; $\eps\to0$,}}\end{equation*}
and, by Lemma~\ref{lemma111}, we see that
\begin{equation*}\label{MEN4}
{\mbox{$\mathcal T_s(\varphi_{\lambda,\eps}-
\chi_{\{u\geq \lambda\}}, \Omega) \to 0$ \; \; as \; $\eps\to0$.}}\end{equation*}
Accordingly, 
the sequence~$\varphi_{\lambda,\eps}$
fulfills the 
hypotheses of Proposition~\ref{prop110}.
It therefore follows that $\chi_{\{u\geq \lambda\} }$ is an
$s$-minimal function, thus proving the desired result in~\eqref{SJkmd-1}.
\\
Now we prove~\eqref{SJkmd-2}.
To this end, we first point out that when $u$ is an
$s$-minimal function, then
\begin{equation}\label{ECsquest}
{\mbox{all level lets $\{ u\geq \lambda\}$ have finite $s$-perimeter in $\Omega$.}}\end{equation}
Indeed, in light of
Lemma~\ref{TGAet}, we know that~$\varphi_{\lambda, \eps} $ is a
minimiser for $\G$ in $\Omega$.  
Then, using the notations in Definition~\ref{defn11} we have that
	\begin{equation}\label{7tegudhfgewhv3irfg}
		\G(\varphi_{\lambda, \eps}, \Omega) \leq  \G (\tilde \varphi_{\lambda, \eps}, \Omega) \leq 2 \Per_s(\Omega, \Rn)
.	\end{equation}
	 Moreover, using Fatou's Lemma (and recalling~\eqref{PER}), we have that
	\bgs{
	\Per_s\left( \left\{ u\geq \lambda\right\} , \Omega\right)  =  \G(\chi_{\{u\geq \lambda\}}, \Omega)
	 \leq \liminf_{\eps \to 0} \G(\varphi_{\lambda, \eps}, \Omega).
	}
{F}rom this and~\eqref{7tegudhfgewhv3irfg} we obtain~\eqref{ECsquest},
as desired.\\
Now we observe that,
for any $F\subset \Rn$ such that $F\cap \Co \Omega= \{ u\geq \lambda\}\cap \Co \Omega$, 
	\[
		 \Per_s\left( \left\{ u\geq \lambda\right\} , \Omega\right) 
		 =  \G(\chi_{\{u\geq \lambda\}}, \Omega) 
		 \leq  \G(\chi_F, \Omega)
		 =\Per_s\left( F, \Omega\right) .
	\]
	This, together with~\eqref{ECsquest}, concludes the proof of the desired result in~\eqref{SJkmd-2}.
\end{proof}

Theorem~\ref{thm211} can be seen as a fractional counterpart
of a classical result stated on page~249 of~\cite{bomby}.

Notice that, thanks to the co-area formula
in Remark~\ref{coareag}, we have the viceversa of Theorem~\ref{thm211}.

\begin{proposition}\label{fgha}
Let $u\in \W^s(\Omega)$. If the set $\{ u \geq \lambda \}$ is $s$-minimal in $\Omega$ for almost every $\lambda \in \R$, then $u\in \W^s(\Omega)$ is an $s$-minimal function
in~$\Omega$. 
\end{proposition}

\begin{proof}
Let
	\[
		\Sigma := \{ \lambda \in \R \; | \; \{u\geq \lambda\}   \mbox{ is not an
$s$-minimal set in }  \Omega \}.
	\]
By assumption, we know that~$ |\Sigma|=0$. 
	Let $v\in \W^s(\Omega)$ be such that~$v=u$ in~$\Co\Omega$. By the co-area formula in Remark~\ref{coareag}, we obtain
	\bgs{
	& \iint_{Q(\Omega)} \left(|u(x)-u(y)|- |v(x)-v(y)|\right) \frac{dx \, dy}{|x-y|^{n+s}} 
	\\
	= &\; \int_{\R \setminus \Sigma} \big( \Per_s\left( \left\{ u\geq \lambda\right\} , \Omega\right)  - \Per_s\left( \left\{ v\geq \lambda\right\} , \Omega\right) \big) \, d\lambda  \leq 0,
	}
	by the $s$-minimality of the set~$\{ u \geq \lambda \}$.
\end{proof}

The claim in Theorem~\ref{STER} is now a direct consequence
of Theorem~\ref{thm211} and
Proposition~\ref{fgha}.

\section{Characteristic functions of $s$-minimal sets and proof of Theorem~\ref{TUTT}}\label{GSmshde23d}

The goal of this section is to prove that
the characteristic functions of $s$-minimal sets are minimisers of $\tilde \G$.
As a matter of fact, a direct consequence of~\eqref{jhgf}
is that characteristic functions of $s$-minimal sets are minimisers of~$\tilde \G$
with respect to characteristic  functions. We prove here the stronger result
in Theorem~\ref{TUTT}, which allows competitors that are not necessarily characteristic functions.

\begin{proof}[Proof of Theorem~\ref{TUTT}]
We see at first that $\mathcal T_s(\chi_E,\Omega)<+\infty$, hence, by~\eqref{CoA con G},
	\begin{equation}\label{VGAJ-AA1}
		 \G(\chi_E,\Omega) = \int_{-\infty}^\infty \Per_s\left( \left\{ \chi_E \geq \lambda\right\},\Omega\right) \, d\lambda.
	\end{equation}
		For any $\lambda \in (0, 1]$, we have that 
			\[ 
			\{ \chi_E \geq \lambda\}=E,
			\]
and, if~$u$ is a competitor as in the statement of Theorem~\ref{TUTT}, 
		\[
			\{u\geq \lambda\} \cap \Co \Omega= \{\chi_{E_0}\geq \lambda\}\cap \Co \Omega=E_0.
		\] 		
	Therefore, by the $s$-minimality of the set~$E$, 
	\begin{equation}\label{9iks-93rjfjj-1}
		\Per_s(\{ \chi_E \geq \lambda\}, \Omega)= \Per_s(E, \Omega) \leq \Per_s(\{ u \geq \lambda\}, \Omega)
		\end{equation}
		whenever $\lambda \in (0, 1]$.
		Furthermore, when $\lambda >1$, then
		\[ 
			\{ \chi_E\geq \lambda\}= \emptyset,
			\]
		and, when $\lambda \leq 0$, we have that
		\[ 
			\{ \chi_E\geq \lambda\}= \Rn.
			\]
			Thus when $\lambda \in(-\infty,0] \cup (1,\infty]$ we get 
			\begin{equation}\label{9iks-93rjfjj-2}
			 \Per_s(\{ \chi_E\geq \lambda\}, \Omega)=0 \leq \Per_s(\{ u\geq \lambda\},\Omega).
			\end{equation}
			It follows from~\eqref{9iks-93rjfjj-1} and~\eqref{9iks-93rjfjj-2} that, for all~$\lambda\in\R$,
$$ \Per_s(\{ \chi_E\geq \lambda\}, \Omega)\leq \Per_s(\{ u\geq \lambda\},\Omega).$$
As a consequence,
			\begin{equation}\label{VGAJ-AA2} 
				\int_{-\infty}^\infty \Per_s(\{ \chi_E\geq \lambda\},\Omega) \, d\lambda
				\leq \int_{-\infty}^\infty \Per_s(\{ u\geq \lambda\},\Omega) \, d\lambda
				= \G(u,\Omega)
				\end{equation}
				using again~\eqref{CoA con G} in the last equality.\\ 
				{F}rom~\eqref{VGAJ-AA1} and~\eqref{VGAJ-AA2}, it follows that
				\[ 
				 \G(\chi_E,\Omega)\leq 
				\G (u,\Omega),
				\]
as desired.\end{proof}

For completeness, we give now some useful characterisations for the finiteness
of the fractional perimeter.

\begin{lemma} \label{CONFIZIC}
Let~$\Omega\subset\R^n$ be a bounded, open set.
Let~$E\subset\R^n$ and~$s\in(0,1)$.
Let
$$ {\mathcal{S}}_E:=\{ (x,x_{n+1})\in\R^n\times\R \;|\; x_{n+1}<\chi_E(x)\}=
\big(E\times(-\infty,1)\big)\cup\big(\Co E\times(-\infty,0)\big).$$
The following conditions are equivalent:
\begin{itemize}
\item[(i)] $\Per_s^{L} (E,\Omega)<+\infty$;
\item[(ii)] $\chi_E\in W^{s,1}(\Omega)$;
\item[(iii)] $\Per_s^{L} ({\mathcal{S}}_E,\Omega\times\R)<+\infty$.
\end{itemize}
\end{lemma}

\begin{proof} The equivalence of~(i) and~(ii) follows directly from~\eqref{PLa}.
The equivalence of~(ii) and~(iii) is the content of Lemma~4.2.6 in~\cite{tesilu}.
\end{proof}

\section{Existence theory and Yin-Yang results:
construction of a function of least $W^{s,1}$-seminorm and
proof of Theorems~\ref{INTAL} and~\ref{thm13}}\label{GA:eucjex}

This section is devoted to the existence theory for the $s$-minimal functions
presented in Definition~\ref{defn11}
and, as a byproduct, to the construction of Yin-Yang minimal surfaces.

We first focus on the proof of Theorem~\ref{thm13},
which will also play a pivotal role in the proof of Theorem~\ref{INTAL}.
For this, we introduce a suitable notation for the ``tail'' near the domain:
namely, for a given open set $\mathcal O \supset \supset \Omega$, we define the ``local tail'' as in~\cite{teolu}, that is
	\begin{equation} \label{truntail}
		\Tail_s\left(\varphi, \mathcal O \setminus \Omega; x\right) := \int_{\mathcal O \setminus \Omega} \frac{ |\varphi(y)|}{|x-y|^{n+s}} \, dy.
	\end{equation}
	Notice that if $\mathcal O=\Rn$, then the ``entire tail''
in formula in~\eqref{coda} can be written
in terms of the ``local tail'', as indeed we have that
	\begin{equation*}
	\mathcal T_s (\varphi, \Omega) = \| \Tail_s(\varphi, \Rn \setminus \Omega; \cdot)\|_{L^1(\Omega)}.
	\end{equation*}

The next result relates the fractional seminorm
with the tail of an $s$-minimal function. It plays a crucial role in the proofs
of Theorems~\ref{INTAL} and~\ref{thm13}.

\begin{theorem}\label{thm51} Let~$d:= \diam (\Omega)$.
There exists $\Theta = \Theta(n,s)>1$ such that if $\varphi \colon \Co \Omega \to \R$ satisfies
\begin{equation}\label{4.1bis} \Tail_s\left(\varphi, \Omega_{\Theta d } \setminus \Omega; \cdot\right) \in L^1(\Omega) \end{equation}
	and  $u\in \W_\varphi^s(\Omega)$ is an
$s$-minimal function in~$\Omega$, then, for any $\lambda \in \R$, it holds that
	\[ 
		\|u- \lambda \|_{W^{s,1}(\Omega)} \leq C_{n,s,d} \| \Tail_s(\varphi - \lambda, \Omega_{\Theta d}\setminus \Omega, \cdot) \|_{L^1(\Omega)}.
	\]
\end{theorem}

\begin{proof}
Notice at first that, for any~$\Theta>1$ and~$\lambda\in\R$,
	\[
	\|\Tail_s\left(\varphi-\lambda, \Omega_{\Theta d } \setminus \Omega; \cdot\right)\|_{L^1(\Omega)}   \leq \|\Tail_s\left(\varphi, \Omega_{\Theta d } \setminus \Omega; \cdot\right) \|_{L^1(\Omega)}+ |\lambda|\, \Per_s(\Omega, \Rn).\]
	Hence, in light of~\eqref{4.1bis}, we find that
	$\Tail_s\left(\varphi-\lambda, \Omega_{\Theta d } \setminus \Omega; \cdot\right) \in L^1(\Omega) $.

We now follow the argument in~\cite[Proposition 4.5.9]{tesilu}, letting 
	\sys[v:=]{& 0 &\mbox{ in } &\Omega,
				\\
				& \varphi -\lambda &\mbox{ in } &\Co \Omega.
				}
Since $u\in\W_\varphi^s(\Omega)$ is an
$s$-minimal function, then $u-\lambda \in \W^s_{\varphi -\lambda }(\Omega)$ is an
$s$-minimal function as well. Thus, since~$v$ is a competitor for~$u-\lambda$,
Definition~\ref{defn11} yields that
	\bgs{
		0\ge&\;
\iint_{Q(\Omega)} \big(|(u(x)-\lambda) -(u(y)-\lambda)|-|v(x)  - v(y)|\big) 
	\frac{dx\, dy}{|x-y|^{n+s}}\\
	=&\;\int_{\Omega}\int_{\Omega} \frac{| u(x)-u(y)| }{|x-y|^{n+s}}
	\,dx\, dy+2
\int_{\Omega}\int_{\Co\Omega}
\big(|u(x)-\varphi(y)|-| \varphi(y)-\lambda|\big) 
	\frac{dx\, dy}{|x-y|^{n+s}},
}
and consequently
\begin{equation}\label{85gfbn-2gvbb}
		[u]_{W^{s,1}(\Omega)} = [u-\lambda ]_{W^{s,1}(\Omega)} \leq 2 
			\int_{\Omega} \left( \int_{\Co \Omega} 
				\frac{|\varphi(y)-\lambda|-|u(x)-\varphi(y)| }{|x-y|^{n+s}}\, dy\right) dx.
	\end{equation}
	Moreover, we have that
\begin{equation*}\label{85gfbn-2gvbb2}\begin{split}
	&  \int_{\Omega} \left( \int_{\Co \Omega} \frac{|\varphi(y)-\lambda|-|u(x)-\varphi(y)| }{|x-y|^{n+s}}\, dy\right) dx
	\\  
	\leq &\; 
	\int_{\Omega} \left( \int_{ \Omega_{\Theta d} \setminus \Omega} \frac{|\varphi(y)-\lambda| }{|x-y|^{n+s}}\, dy\right) dx 
	- \int_{\Omega} \left( \int_{ \Omega_{\Theta d} \setminus \Omega} \frac{|u(x)-\varphi(y)| }{|x-y|^{n+s}}\, dy\right) dx 
	\\
	&\;+ 
	\int_{\Omega} \left( \int_{\Co \Omega_{\Theta d} } \frac{|u(x)-\lambda|}{|x-y|^{n+s}}\, dy\right) dx
	\\ \leq &\; 
	\int_{\Omega} \left( \int_{ \Omega_{\Theta d} \setminus \Omega} \frac{|\varphi(y)-\lambda| }{|x-y|^{n+s}}\, dy\right) dx 
	- \int_{\Omega} \left( \int_{ \Omega_{\Theta d} \setminus \Omega} \frac{|u(x)-\varphi(y)| }{|x-y|^{n+s}}\, dy\right) dx 
	\\
&\;+ 
	\int_{\Omega} \left( \int_{\Co B_{\Theta d}(x) } \frac{|u(x)-\lambda|}{|x-y|^{n+s}}\, dy\right) dx
	\\
	\leq &\; \|\Tail_s(\varphi-\lambda,\Omega_{\Theta d} \setminus \Omega; \cdot)\|_{L^1(\Omega)} +\frac{c_n(\Theta d)^{-s}}s \|u-\lambda\|_{L^1(\Omega)}
	-
	\int_{\Omega} \left( \int_{ \Omega_{\Theta d} \setminus \Omega} \frac{|u(x)-\varphi(y)| }{|x-y|^{n+s}}\, dy\right) dx ,
	\end{split}\end{equation*}
for some positive constant~$c_n$.
As a consequence, it follows 
in~\eqref{85gfbn-2gvbb} that 
\begin{equation}\label{85gfbn-2gvbb3}\begin{split}
&	[u-\lambda]_{W^{s,1}(\Omega)} 
		+\int_{\Omega} \left( \int_{ \Omega_{\Theta d} \setminus \Omega} \frac{|u(x)-\varphi(y)| }{|x-y|^{n+s}}\, dy\right) dx  
			\\ &\qquad \leq  
			2   \|\Tail_s(\varphi-\lambda,\Omega_{\Theta d} \setminus \Omega; \cdot)\|_{L^1(\Omega)} 
			   +
			   \frac{c_n(\Theta d)^{-s}}s \|u-\lambda\|_{L^1(\Omega)}.
	\end{split}\end{equation}
	Now, using a fractional Poincar\'e inequality (see~\cite{teolu} or~\cite[Lemma D.1.6]{tesilu}) we have that
\begin{equation}\label{utyrgfvccbn}
	 \|u-\lambda\|_{L^1(\Omega)}
	 	= \|u-\lambda-v\|_{L^1(\Omega)} 
		 \leq \frac{\left(\diam(\Omega_{ d} )\right)^{n+s} }{| \Omega_{ d} \setminus 
\Omega |}
		 \int_{\Omega}\left(  \int_{  \Omega_{ d} \setminus \Omega} 
\frac{|u(x)-\lambda |}{|x-y|^{n+s} }dy\right) dx.
\end{equation}
Furthermore, by the triangle inequality, for $y\in \Omega_d \setminus \Omega$,
$$ |u(x)-\lambda|\le |u(x)-u(y)|+|u(y)-\lambda|=  |u(x)-u(y)|+|\varphi(y)-\lambda|.$$
Therefore, 
\bgs{
\int_{\Omega}\left(  \int_{  \Omega_{ d} \setminus \Omega} 
\frac{|u(x)-\lambda |}{|x-y|^{n+s} }dy\right) dx
	\le&\; 
\int_{\Omega}\left(  \int_{  \Omega_{ d} \setminus \Omega} 
\frac{|u(x)-u(y)|}{|x-y|^{n+s} }dy\right) dx +
\int_{\Omega}\left(  \int_{  \Omega_{ d} \setminus \Omega} 
\frac{|\varphi(y)-\lambda |}{|x-y|^{n+s} }dy\right) dx
\\
=&\;
\int_{\Omega}\left(  \int_{  \Omega_{ d} \setminus \Omega} 
\frac{|u(x)-u(y)|}{|x-y|^{n+s} }dy\right) dx +
\|\Tail_s(\varphi-\lambda, \Omega_d \setminus \Omega; \cdot)\|_{L^1(\Omega)}.
}
Hence, since~$\Theta>1$,
\begin{equation}\label{jfhgrtuhgbueuhe089u8y6u}
\int_{\Omega}\left(  \int_{  \Omega_{ d} \setminus \Omega} 
\frac{|u(x)-\lambda |}{|x-y|^{n+s} }dy\right) dx\le 
\int_{\Omega}\left(  \int_{  \Omega_{ \Theta d} \setminus \Omega} 
\frac{|u(x)-u(y)|}{|x-y|^{n+s} }dy\right) dx +
\|\Tail_s(\varphi-\lambda, \Omega_{\Theta d} \setminus \Omega; \cdot)\|_{L^1(\Omega)}.
\end{equation}
Also, 
\begin{equation}\label{jfhgrtuhgbueuhe089u8y6u-2}
|\Omega_d \setminus \Omega| \geq c\, d^n,
\end{equation}
for some~$c>0$, possibly depending on~$n$ and~$\Omega$.

By inserting~\eqref{jfhgrtuhgbueuhe089u8y6u}
and~\eqref{jfhgrtuhgbueuhe089u8y6u-2} into~\eqref{utyrgfvccbn}, we obtain that
\begin{equation}\label{utyrgfvccbn3}
	 \|u-\lambda\|_{L^1(\Omega)}
		 \leq C_n d^s \left[
		 \int_{\Omega} \left(  \int_{  \Omega_{ \Theta d} \setminus \Omega}
 \frac{|u(x)-u(y)|}{|x-y|^{n+s} }dy\right) dx 
		 + \|\Tail_s(\varphi-\lambda, \Omega_{\Theta d} \setminus \Omega; \cdot)
\|_{L^1(\Omega)} \right]
,
\end{equation}
for some~$C_n$ possibly depending also on~$\Omega$.
\\
Summing up~\eqref{85gfbn-2gvbb3} and~\eqref{utyrgfvccbn3}, we deduce that
\bgs{
&\; 	[u-\lambda]_{W^{s,1}(\Omega)} 
		+\int_{\Omega} \left( \int_{ \Omega_{\Theta d} \setminus \Omega} \frac{|u(x)-\varphi(y)| }{|x-y|^{n+s}}\, dy\right) dx  
	+\frac{d^{-s}}{C_n}	\|u-\lambda\|_{L^1(\Omega)}\\
\leq & \;
			3   \|\Tail_s(\varphi-\lambda,\Omega_{\Theta d} \setminus \Omega; \cdot)\|_{L^1(\Omega)} 
			   +
			   \frac{c_n(\Theta d)^{-s}}s \|u-\lambda\|_{L^1(\Omega)}
\\
&\;  	
+	 \int_{\Omega} \left(  \int_{  \Omega_{ \Theta d} \setminus \Omega}
 \frac{|u(x)-u(y)|}{|x-y|^{n+s} }dy\right) dx .
}
Consequently, simplifying one term, 
it follows that
		 \bgs{
		 	\frac{d^{-s}}{C_n} \|u-\lambda\|_{L^1(\Omega)} \left(1- C_n\, c_n \frac{\Theta^{-s}}{s} \right)
		 	+ [u-\lambda]_{W^{s,1}(\Omega)} 
		 	\leq 3 \|\Tail_s(\varphi-\lambda, \Omega_{\Theta d} \setminus \Omega; 
\cdot)\|_{L^1(\Omega)} .
		 	} 
Choosing $\Theta$ large enough in the latter estimate, we obtain the desired result.
\end{proof}	  

We observe that
Theorem~\ref{thm51} entails the following simple, but interesting, consequence:

\begin{corollary}\label{cor52}
Let $\Theta, d$ be as in Theorem~\ref{thm51}. If $\varphi \colon \Co \Omega \to \R$ is such that $\varphi = \lambda$ in $\Omega_{\Theta d} \setminus \Omega$, and  $u\in \W_\varphi^s(\Omega)$ is an
$s$-minimal function in~$\Omega$, then $u=\lambda $ almost everywhere in $\Omega$.
\end{corollary}



Theorem~\ref{thm13} is now a direct consequence of
Theorem~\ref{TUTT} and Corollary~\ref{cor52}.\medskip

To complement the picture given in Theorem~\ref{thm13},
we take this opportunity to stress the importance of the regularity
of the domain on the filling and emptying phenomena of nonlocal minimal surfaces.

\begin{proposition} \label{INSJ:PRO}
Let~$s\in(0,1)$.
Let~$\Omega\subset\R^n$ be a bounded, open set, with~$0\in\partial\Omega$.
Let~$r>0$ and
assume that~$\Omega\setminus B_{r/2}$
has Lipschitz boundary, and that
\begin{equation}\label{infid}
\Per_s(\Omega,B_r)=+\infty.\end{equation}
Let~$E_0\subset\Co\Omega$ be such that
\begin{equation}\label{siusa}
B_r\setminus\Omega\subset E_0.\end{equation}
Let~$E$ be an~$s$-minimal set in $\Omega$ with respect to~$E_0$.
Then,
\begin{equation}\label{ENONAx} E\cap\Omega\cap B_r\ne\emptyset.\end{equation}
\end{proposition}

\begin{proof} 
First of all, we point out that the statement in Proposition~\ref{INSJ:PRO} is non-void,
since an~$s$-minimal set does exist.
For this, let~$F_0:=B_r\cup E_0$.
We observe that
\bgs{
		\Per_s(F_0,\Omega)=&\;
	\int_{B_r\cap\Omega}\int_{ \Omega\setminus B_r} \frac{dx\,dy}{|x-y|^{n+s}}+
\int_{B_r\cap\Omega}\int_{(\Co E_0)\setminus \Omega} \frac{dx\,dy}{|x-y|^{n+s}} +
\int_{\Omega\setminus B_r}\int_{E_0} \frac{dx\,dy}{|x-y|^{n+s}} 
\\ 
	\leq &\; 
\int_{B_r}\int_{ \Co B_r} \frac{dx\,dy}{|x-y|^{n+s}}+
\int_{B_r}\int_{\Co B_r} \frac{dx\,dy}{|x-y|^{n+s}}+\int_{\Omega\setminus B_{r/2}}\int_{\Co (\Omega\setminus B_{r/2})} \frac{dx\,dy}{|x-y|^{n+s}}\\
	\leq&\;
			2\Per_s(B_r,\R^n)+\Per_s(\Omega\setminus B_{r/2},\Rn )
	\\
	<&\; +\infty.
}
This and the fact that~$F_0\setminus\Omega=E_0$ give that~$F_0$
is an admissible competitor with finite nonlocal perimeter.
Then, the existence of an $s$-minimal set~$E$ with datum~$E_0$
is warranted by Theorem~1.9
of~\cite{approxLuk}.

Hence, to complete the proof of Proposition~\ref{INSJ:PRO},
it remains to show~\eqref{ENONAx}.
To this end, we argue towards a contradiction and we assume the converse, namely that
\begin{equation}\label{siusa2}
E\cap\Omega\cap B_r=\emptyset.\end{equation}
Then, using~\eqref{siusa} and~\eqref{siusa2}, we have that
\bgs{
		\Per_s(E,\Omega)\geq&\; 
\int_{\Omega\setminus E}\int_{E\setminus\Omega} \frac{dx\,dy}{|x-y|^{n+s}}\\ \geq&\; 
\int_{\Omega\cap B_r}\int_{E_0} \frac{dx\,dy}{|x-y|^{n+s}}\\
		\geq &\; 
\int_{\Omega\cap B_r}\int_{B_r\setminus\Omega} \frac{dx\,dy}{|x-y|^{n+s}}\\
		= &\; 
\Per_s(\Omega,B_r)-
\int_{\Omega\cap B_r}\int_{\Co B_r\setminus\Omega} \frac{dx\,dy}{|x-y|^{n+s}}-
\int_{ B_r\setminus\Omega}\int_{\Co B_r\cap\Omega} \frac{dx\,dy}{|x-y|^{n+s}}
		\\
			\geq &\; 
\Per_s(\Omega,B_r)-2\Per_s(B_r,\R^n).
}
This and~\eqref{infid} give that~$\Per_s(E,\Omega)=+\infty$,
which is in contradiction with the $s$-minimality of
the set~$E$, and thus it completes the proof
of~\eqref{ENONAx}.
\end{proof}

An explicit construction of a domain~$\Omega$ satisfying the assumptions
of Proposition~\ref{INSJ:PRO} can be obtained as follows:
in Example~3 of~\cite{asympt1} (see in particular Section~3.10 of~\cite{asympt1})
one constructs a set~$E_\star\subset(0,1)$ which is the union of countably many
intervals accumulating to the origin and such that~$\Per_s(E_\star,(-r,r))=+\infty$
for all~$s\in(0,1)$ and all~$r\in(0,1)$. Then, we can define
$$\Omega:=\left(
\left(E_\star\times\left(-\infty,\frac{1}{10}\right)\right)
\cup\Big(\Co E_\star\times(-\infty,0)\Big)\right)\cap B_2,$$
and we deduce from Lemma~\ref{CONFIZIC}
(used here with~$r:=1/4$) that~$\Per_s(\Omega,B_{1/4})=+\infty$ for all~$s\in(0,1)$.
\medskip

Now, we focus on the proof of Theorem~\ref{INTAL}.
To this end, we state a general maximum principle.

\begin{theorem}\label{MAXPL}
There exists $\Theta = \Theta(n,s)>1$ such that, denoting
$d:= \diam (\Omega)$, the following statement holds true.

If $u\in \W^s(\Omega)$ is an $s$-minimal function in~$\Omega$, then
\begin{equation}\label{INund:sd9i331} \sup_{\Omega} u \leq \sup_{\Omega_{\Theta d}\setminus \Omega} u\qquad{\mbox{and}}\qquad \inf_{\Omega} u \geq \inf_{\Omega_{\Theta d}\setminus \Omega} u.\end{equation}
\end{theorem}

\begin{proof} We prove the first inequality in~\eqref{INund:sd9i331},
since the second inequality can be proved similarly. To this end,
we suppose that $\sup_{\Omega_{\Theta d}\setminus \Omega} u <+\infty$ (otherwise, there is nothing to prove). We let 
	\[
	 \lambda > \sup_{\Omega_{\Theta d}\setminus \Omega} u,
	\]
	then
	\[
		 \{u\geq \lambda\} \cap \left({\Omega_{\Theta d}\setminus \Omega}\right) =\emptyset.
	\]
	As a consequence, by Theorems~\ref{thm211} and~\ref{thm13} it follows that 
	\[ 
		\{ u\geq \lambda\} \cap \Omega = \emptyset,
	\]
	therefore 
	\[
		\sup_{\Omega} u \leq \lambda 
	\]
	for any $\lambda > \sup_{\Omega_{\Theta d}\setminus \Omega} u$.
This proves the first inequality in~\eqref{INund:sd9i331}.
\end{proof}

Our goal is now to construct ``maximal'' and ``minimal''
$s$-minimal sets (this is needed since $s$-minimal
sets are not necessarily unique, as discussed in detail in Theorem~\ref{NONUN}).
For this maximal/minimal construction of $s$-minimal sets, we start by noticing that the $s$-minimality of sets
is preserved under union and intersections.

\begin{lemma}\label{lemma52}
Let $E_0 \subset \Co \Omega$, and let $E,F\subset \Rn$ be $s$-minimal sets
in $\Omega$ with respect to $E_0$. Then $E\cup F$
and~$E \cap F$ are also $s$-minimal sets
in~$\Omega$
with respect to $E_0$.  
\end{lemma}

\begin{proof} By a direct computation, we see that
	\bgs{
	|\chi_{E \cup F}(x) -\chi_{E\cup F}(y)| +|\chi_{E \cap F}(x) -\chi_{E\cap F}(y)| \leq |\chi_{E }(x) -\chi_{E}(y)| +|\chi_{ F}(x) -\chi_{ F}(y)|.
	}
Hence, by~\eqref{PER} we get that 
	\eqlab{ \label{minmaxx}
	\Per_s(E\cap F,\Omega) + \Per_s(E \cup F,\Omega) \leq \Per_s(E,\Omega)+\Per_s(F,\Omega) 
	.}
	We stress that~\eqref{minmaxx} is in fact valid for all sets~$E$ and~$F$ (we have not used
	here any minimality condition or the fact that~$E\setminus\Omega=F\setminus\Omega)$.\\
Also, since $(E\cap F)\setminus\Omega=E_0$, by the $s$-minimality of the set~$F$ we obtain 
	\[	
		\Per_s(F,\Omega) 
		\leq \Per_s(E \cap F,\Omega).
	\]
{F}rom this and~\eqref{minmaxx} it follows that 
	\[ \Per_s(E\cup F,\Omega)\leq \Per_s(E,\Omega).\]
	As a result, since~$(E\cup F)\setminus\Omega=E_0$
and~$
E$ is an~$s$-minimal set,
we find that~$E\cup F$ is an
$s$-minimal set as well. The same holds for $E \cap F$. 
\end{proof}

The result in Lenma~\ref{lemma52}
allows us to introduce the notion of ``maximal'' and ``minimal''
$s$-minimal sets, by arguing as follows:

\begin{proposition}\label{Prop53}
Let $E_0 \subset \Co \Omega$ and
	\bgs{ \label{bigcup}
		\mathcal F: = \left\{ E \subset \Rn \; | \; E \mbox{ is an } s\mbox{-minimal
set in }\Omega, \: E\setminus \Omega =E_0  \right\}. 
	}
Then there exists a unique set $E\in \mathcal F$ with maximum volume inside $\Omega$ and
	\begin{equation}\label{MAX:SE1}
		E= \bigcup_{F\in \mathcal F} F.
	\end{equation}
Moreover, there exists a unique set $E\in \mathcal F$ with minimum volume inside $\Omega$ and
	\begin{equation}\label{MAX:SE2}
		E = \bigcap_{F\in \mathcal F} F.
	\end{equation}
\end{proposition}

\begin{proof}
Let 
	\begin{equation}\label{68tuegef238392tfegvbsjwgf} 
		M:= \sup_{E\in \mathcal F} |E\cap \Omega|
	\end{equation}
	and let $E_k \in \mathcal F$ be such that
	\[
		 |E_k\cap  \Omega | \to M \quad \mbox{ for } k \to+\infty.
	\]
We notice that, in light of~\eqref{PLa},
	\[ 
			[\chi_{E_k\cap \Omega}]_{W^{s,1}(\Omega)} =2 	\Per^L_s(	E_k
\cap \Omega,\Omega)
= 2 	\Per^L_s(	E_k,\Omega)
	.\]
		As a result, we conclude that
		\[ \|\chi_{E_k\cap \Omega}\|_{W^{s,1}(\Omega)} \leq |\Omega| + 2 
		\Per^L_s(	E_k,\Omega)\le |\Omega| + 2 
		\Per_s(	E_k,\Omega),
		\]
which is bounded uniformly in~$k$, since every set~$E_k$
is~$s$-minimal in~$\Omega$.

		Hence, by the compact embedding of $W^{s,1}(\Omega)$ into $L^1(\Omega)$, there exists $E \subset \Rn$ with $E\setminus \Omega=E_0$ such that (up to a subsequence)
		\[
		\chi_{E_k} \to\chi_{ E} \; \mbox{ in } \; L^1(\Omega) \quad \mbox{ for } k \to+\infty.
		\]
		Therefore, recalling~\eqref{68tuegef238392tfegvbsjwgf},
		\[
		|E\cap \Omega|=\lim_{k \to+\infty} |E_k\cap \Omega| = M.
		\]
		Moreover, by Fatou's Lemma,
		\[
		 \Per_s(E,\Omega) \leq \liminf_{k \to+\infty} \Per_s(E_k,\Omega)\leq \Per_s(F,\Omega) \qquad \mbox{ for any } F\setminus \Omega =E_0,
		\] 
which establishes that
the set~$E$ is $s$-minimal.

Accordingly, to complete the proof of Proposition~\ref{Prop53},
it remains to check that such an~$s$-minimal set is
unique (up to null sets) and to establish~\eqref{MAX:SE1}. We start
with the uniqueness statement. To this end, 
denoting by~``$\triangle$''
the symmetric difference between two sets,
we
suppose that
there exist $E, F\in \mathcal F$, such that
\begin{equation}\label{Pjwdnd-khqsdsd-1}|E\cap \Omega|=
|F\cap \Omega|=M \end{equation}
and
\begin{equation*}
|E\triangle F|>0.\end{equation*}
In particular, up to exchanging the roles of~$E$ and~$F$, we can suppose that  
	\[ |E\setminus F|>0.\]
In addition, the sets~$E$ and~$F$ coincide outside~$\Omega$,
and therefore
$$ E\setminus F=(E\setminus F)\cap\Omega.$$
Thus we have that 
\bgs{
|E\cap\Omega|=&\; |E\cap F\cap\Omega|+
|(E\setminus F)\cap\Omega|\\
=&\; |E\cap F\cap\Omega|+
|E\setminus F|\\>&\; |E\cap F\cap\Omega|.
}
On this account, we see that
\bgs{
|(E\cup F)\cap\Omega| =&\;|(E\cap\Omega)\cup(F\cap\Omega)|\\
=&\; |E\cap\Omega|+|F\cap\Omega|-|E\cap F\cap\Omega|\\>&\;
|F\cap\Omega|.
}
This and~\eqref{Pjwdnd-khqsdsd-1} give that
\begin{equation}\label{Pjwdnd-khqsdsd-44}
|(E\cup F)\cap\Omega|>M.
\end{equation}
On the one hand, by Lemma~\ref{lemma52},
we have 
that
\begin{equation*}
(E\cup F) \in \mathcal F.\end{equation*}
Hence, in light of~\eqref{68tuegef238392tfegvbsjwgf}, we see that
$$ M\ge |(E\cup F)\cap\Omega|,$$
and this provides a contradiction with~\eqref{Pjwdnd-khqsdsd-44}.

We now focus on the proof of~\eqref{MAX:SE1}. For this, 
we observe that 
		 \begin{equation}\label{oho7088901}
		 	E\subset \bigcup_{F\in \mathcal F} F. 
		 \end{equation}
We claim that
\begin{equation}\label{oho708890}
 \bigcup_{F\in \mathcal F} F\subset E.
\end{equation}
To prove it we suppose, by contradiction, that there exists~$F\in \mathcal F$
such that~$|F \setminus E|>0$. Therefore, 
		 \[
		 \left|\left(E\cup F\right) \cap \Omega\right| 
		 = |E\cap \Omega|+\left|\left(F\setminus E\right) \cap \Omega \right| >|E\cap \Omega|.
		 \] This and~\eqref{Pjwdnd-khqsdsd-1} yield that
\begin{equation}\label{0-qtweg38484}
	 \left|\left(E\cup F\right) \cap \Omega\right|>M.
\end{equation}
Moreover, in virtue of Lemma~\ref{lemma52},
we know that~$E\cup F$ is still an $s$-minimal set, and therefore~$(E\cup F)\in \mathcal F$.
As a consequence, by~\eqref{68tuegef238392tfegvbsjwgf}, we conclude that~$M\ge
\left|\left(E\cup F\right) \cap \Omega\right|$. This is in contradiction with~\eqref{0-qtweg38484},
and therefore the claim in~\eqref{oho708890} is established.

{F}rom~\eqref{oho7088901} and~\eqref{oho708890} we complete the proof of~\eqref{MAX:SE1}.
The proof of the second part of Proposition~\ref{Prop53} can be done in a similar way.
\end{proof}

We take this opportunity to point out, as a simple but interesting consequence
of Proposition~\ref{Prop53}, that external data that are rotationally symmetric
always admit an $s$-minimal set which is rotationally symmetric as well
(more generally, external data with a given symmetry
always admit an $s$-minimal set with the same symmetry).

\begin{corollary} Let~${\mathcal{Z}}$ be a family of isometries.
Assume that, for every~${\mathcal{S}}\in {\mathcal{Z}}$, we have that
\begin{equation}\label{98ry34}
{\mathcal{S}}(\Omega)=\Omega.\end{equation}
Let~$E_0 \subset \Co \Omega$ and suppose that
\begin{equation}\label{98ry34-1}{\mbox{${\mathcal{S}}(E_0)=E_0\qquad$
for every~${\mathcal{S}}\in {\mathcal{Z}}$.}}
\end{equation}
Then, there exists at least a set~$E$ which is~$s$-minimal in~$\Omega$
with respect to $E_0$,
and such that
$${\mbox{${\mathcal{S}}(E)=E\qquad$ for every~${\mathcal{S}}\in {\mathcal{Z}}$.}}$$
\end{corollary}

\begin{proof} We take~$E$ as in~\eqref{MAX:SE1} (we could also take~$E$
as in~\eqref{MAX:SE2}, and we would obtain the same conclusion).
Let~${\mathcal{S}}\in{\mathcal{Z}}$.
Then, we can write~${\mathcal{S}}(x)=Ax+a$
for all~$x\in\R^n$, for some orthogonal matrix~$A$
and some~$a\in\R^n$, see e.g. Theorems~20.7 and~20.8 in~\cite{sernesi2000geometria}.
In particular,
\begin{equation}\label{532}
|\det D{\mathcal{S}}|=|\det A|=1.
\end{equation}
Moreover, we observe that if~$F\in \mathcal F$ then the nonlocal perimeter of~${\mathcal{S}}(F)$
is the same as the one of~$F$, namely
\begin{equation}\label{KA:mwer49545}
\begin{split}
	& \Per_s\left(\mathcal{S}(F),\Omega\right)
		\\
		=\,&\frac12\int_\Omega\int_\Omega
\frac{|\chi_{{\mathcal{S}}(F)}(x) -\chi_{{\mathcal{S}}(F)}(y)|}{|x-y|^{n+s}} \, dx \, dy+
\int_{\Omega}\int_{\Co \Omega} \frac{|\chi_{{\mathcal{S}}(F)}(x) -\chi_{{\mathcal{S}}(F)}(y)|}{|x-y|^{n+s}} \, dx \, dy
	\\
	=\,&
\frac12 \int_{{\mathcal{S}}^{-1}(\Omega)}\int_{{\mathcal{S}}^{-1}(\Omega)}
\frac{|\chi_{F}(\eta) -\chi_{F}(\xi)|}{|\eta-\xi|^{n+s}} \, d\eta \, d\xi+
\int_{{{\mathcal{S}}^{-1}(\Omega)}}\int_{{{\mathcal{S}}^{-1}(\Co\Omega)}} \frac{|\chi_{F}(\eta) -\chi_{F}(\xi)|}{|\eta-\xi|^{n+s}} \, d\eta \, d\xi
	\\
	=\,&
\frac12 \int_{ \Omega}\int_{\Omega}
\frac{|\chi_{F}(\eta) -\chi_{F}(\xi)|}{|\eta-\xi|^{n+s}} \, d\eta \, d\xi+
\int_{ \Omega }\int_{ \Co\Omega} \frac{|\chi_{F}(\eta) -\chi_{F}(\xi)|}{|\eta-\xi|^{n+s}} \, d\eta \, d\xi\\
=\,&\Per_s (F,\Omega)
,\end{split}
\end{equation}
thanks to~\eqref{PLa}, \eqref{98ry34} and~\eqref{532},
where the substitutions~$\eta:={\mathcal{S}}^{-1}(x)$ and~$\xi:={\mathcal{S}}^{-1}(y)$
have been used.

Also, from~\eqref{98ry34-1} we know that~${\mathcal{S}}(F\setminus \Omega)={\mathcal{S}}(E_0)=E_0$. This and~\eqref{KA:mwer49545}
give that~${\mathcal{S}}(F)\in\mathcal F$.

As a consequence, from~\eqref{MAX:SE1},
$$ E= \bigcup_{F\in \mathcal F} F= \bigcup_{F\in \mathcal F} {\mathcal{S}}(F)=
{\mathcal{S}}(E),$$
as desired.
\end{proof}

Now we give an auxiliary existence result.

\begin{theorem}\label{exlev}
There exists $\Theta = \Theta(n,s)>1$ such that, denoting
$d:= \diam (\Omega)$, the following statement holds true.
If $\varphi \colon \Co \Omega \to \R$ is such that
	\begin{equation}\label{PHIA}
		\varphi \in L^\infty(\Omega_{\Omega_{\Theta d} \setminus \Omega}) ,
	 \end{equation}
	then there exists an $s$-minimal function~$u\in \W_\varphi^s(\Omega)$. 
\end{theorem}

\begin{proof}
 For any~$t\in\R$, we consider the set 
 \begin{equation}\label{37893265869tgf65ytgf386ytgvkj}
 	\mathcal E_t := \{ x\in \Co \Omega \; | \; \varphi(x) \geq t\}
 \end{equation}
 and we let $E_t$ be the $s$-minimal set of maximum volume in $\Omega$ with respect to the exterior data $\mathcal E_t$ (which exists, thanks to Proposition~\ref{Prop53}, see in particular~\eqref{MAX:SE1}).
 
  We prove now, inspired by~\cite[Lemma 3.4]{SWZ}, that 
  \eqlab{ \label{tit} {\mbox{if $
 		\tau <t $, then $ E_t \subset E_\tau$, up to null sets. }}
  }
To this end, we observe that, if~$\tau <t$ then~$\mathcal E_t \subset \E_\tau$. As a consequence,
  \[
  	(E_t \cap E_\tau) \setminus \Omega = \mathcal E_t \cap \mathcal \E_\tau = \E_t \qquad{\mbox{and}}
  	\qquad 
  	(E_t \cup E_\tau) \setminus \Omega = \mathcal E_t \cup \mathcal \E_\tau = \E_\tau.
  \]
  These observations give that \begin{equation}\begin{split}\label{MNAsidwq45}
{\mbox{the set~$E_t \cap E_\tau$ is a competitor for~$E_t$,
and the set~$E_t \cup E_\tau$ is a competitor for~$E_\tau$.}} \end{split}\end{equation}
  Hence,
  the minimality  of~$E_t$ and~$E_\tau$ implies that
  \begin{equation}\label{AKsmd:945}
  	\Per_s(E_t,\Omega) \leq\Per_s(E_t \cap E_\tau,\Omega) \qquad{\mbox{and}}\qquad 
  	\Per_s(E_\tau,\Omega) \leq\Per_s(E_t \cup E_\tau,\Omega).
  \end{equation}
Thus, using~\eqref{minmaxx} and the first inequality in~\eqref{AKsmd:945}, we see that
\bgs{
	\Per_s(E_t,\Omega) +\Per_s(E_t\cup E_\tau,\Omega)
	\leq&\; \Per_s(E_t \cap E_\tau,\Omega)+
\Per_s(E_t \cup E_\tau,\Omega)
\\	
\le&\; \Per_s(E_t,\Omega) +\Per_s(E_\tau,\Omega),
}
and accordingly, 
\begin{equation*}
\Per_s(E_t\cup E_\tau,\Omega)\le \Per_s(E_\tau,\Omega).\end{equation*}
This and the second inequality in~\eqref{AKsmd:945} yield that
  \[
  	\Per_s(E_\tau,\Omega) =\Per_s(E_t \cup E_\tau,\Omega).
  \]
Consequently, exploiting the minimality of the set~$E_\tau$ and~\eqref{MNAsidwq45},
we conclude that also~$E_t \cup E_\tau$ is an $s$-minimal set in~$\Omega$.
Thus, since~$E_\tau$ is the $s$-minimal set with maximum volume,
necessarily we have that~$|E_t \cup E_\tau| 	=| E_\tau|$, and then
$$ |E_t\setminus E_\tau|=|E_t \cup E_\tau| 	-| E_\tau|=0,$$
which proves~\eqref{tit}.
\\
Now, we define $u\colon \Rn \to \R\cup\{-\infty,+\infty\}$ as
  \begin{equation} \label{8ihrh4h55}u(x):=
  \begin{cases} \varphi(x),& \mbox{ for } x\in \Co \Omega,
  \\
 			\sup\big\{ t \;|\;
 			x\in \overline{ E_t}\big\},
 			 & \mbox{ for } x\in \Omega.
  \end{cases}\end{equation}
  We claim that
  \begin{equation}\label{QWE}
  u\in L^\infty(\Omega).
  \end{equation}
To prove this, 
we take~$\Theta$ as in Theorem~\ref{thm13} and we point out that,
if $t \leq \displaystyle\inf_{\Omega_{\Theta d} \setminus \Omega} \varphi$, then 
  \[ 
 \mathcal E_t  \cap \left(\Omega_{\Theta d} \setminus \Omega\right) = \Omega_{\Theta d} \setminus \Omega.
  \]
Therefore, according to Theorem~\ref{thm13}, we have that 
  \[
  	E_t \cap \Omega = \Omega.
  \]  
As a consequence, for any~$x\in\Omega$, we have that~$x\in \overline{ E_t}$
for all~$t \leq \displaystyle\inf_{\Omega_{\Theta d} \setminus \Omega} \varphi$.
Accordingly, recalling~\eqref{PHIA}, we conclude that
\begin{equation} \label{0i73ryjh8756bfcb856hgnbv}\sup\big\{ t \;|\;
 			x\in \overline{ E_t}\big\}\ge \inf_{\Omega_{\Theta d} \setminus \Omega} \varphi>-\infty.\end{equation}
In the same way one proves that, for any~$x\in\Omega$,
\begin{equation*}\sup\big\{ t \;|\;
 			x\in \overline{ E_t}\big\}\le\sup_{\Omega_{\Theta d} \setminus \Omega} \varphi<+\infty.
\end{equation*}
{F}rom this and~\eqref{0i73ryjh8756bfcb856hgnbv}, we obtain~\eqref{QWE}, as desired.
As a result,
since $\Omega $ is bounded, we also find that
\begin{equation}\label{JSn0s884-293iedfjfj}
u\in L^1(\Omega).\end{equation} 
  Now, in view of~\eqref{37893265869tgf65ytgf386ytgvkj} and~\eqref{8ihrh4h55}, we see that, for any $t\in \R$,
\begin{equation}\label{jrugh694943}
  	\{ u \geq t\} \setminus \Omega =\{ \varphi \geq t\} \setminus \Omega= \mathcal E_t=E_t\setminus \Omega.
  	\end{equation}
As a result,
\begin{equation}\label{TGBOgkdfi55797}
  	\{ u \geq t\} \setminus E_t =\left(
\{ u \geq t\} \setminus E_t\right)\cap\Omega.  	\end{equation}
Another consequence of~\eqref{8ihrh4h55} is that
  	\begin{equation}\label{yth7a344356cP}
  \left(\overline{ E_t }\cap \Omega\right) \subset \{ u\geq t\} .
  	\end{equation}
  	Furthermore, in light of~\eqref{JSn0s884-293iedfjfj} and Lemma~\ref{poi}, we know that,
  	for almost any $t\in\R$,
  \begin{equation}\label{A2iskxmc9t4hgbP} |u^{-1}(t)\cap \Omega|=0.\end{equation}
In addition, since the set~$E_t$ is $s$-minimal in $\Omega$, 
from Corollary~4.4(i) in~\cite{nms} it follows that, for every~$\Omega'\subset\subset\Omega$,
$$ \mathcal H^{n-s}\big((\partial E_t)\cap \Omega'\big)<+\infty,$$
and, as a result,
$$\big|(\partial E_t)\cap \Omega'\big|=0.$$
For this reason, we obtain that
$$\big|(\partial E_t)\cap \Omega\big|=0.$$
Consequently, recalling~\eqref{TGBOgkdfi55797},
we deduce that
\begin{equation}\label{6648jfLLAPS}
\big|\{ u \geq t\} \setminus E_t \big|=\big|(
\{ u \geq t\} \setminus \overline{E_t})\cap\Omega\big|.\end{equation}
We also claim that
\begin{equation}\label{78-u02oj4756ythn}
  	\left(\{ u\geq t\} \setminus \overline{ E_t}\right) \cap \Omega \;\,\subset\;\, u^{-1}(t)\cap \Omega.
\end{equation}
To check this, we take~$p$ belonging to the set on the left hand side of~\eqref{78-u02oj4756ythn}
and we suppose, towards a contradiction, that
\begin{equation}\label{78-u02oj4756ythn2}
u(p)>t.\end{equation}
We stress that~$p\in\Omega$, hence, by~\eqref{8ihrh4h55} and~\eqref{78-u02oj4756ythn2},
\begin{equation}\label{78-u02oj4756ythn3} u(p)=\sup\big\{ \vartheta \;|\;
 			p\in \overline{ E_\vartheta}\big\}.
\end{equation}
We also recall that~$p\not\in \overline{ E_t}$. Hence, by~\eqref{tit},
we know that~$p\not\in \overline{ E_\tau}$ for all~$\tau\ge t$.
This and~\eqref{78-u02oj4756ythn3} yield that~$u(p)\le t$.
But this inequality is in contradiction with~\eqref{78-u02oj4756ythn2},
and therefore the proof of~\eqref{78-u02oj4756ythn} is complete.

Then, gathering the results in~\eqref{A2iskxmc9t4hgbP},
\eqref{6648jfLLAPS}
and~\eqref{78-u02oj4756ythn}, we deduce that,
for almost every~$t\in \R$,
\begin{equation*}\label{amt}\begin{split}
\big|\{ u\geq t \} \setminus E_t \big| 
=\; 
\big|(
\{ u \geq t\} \setminus \overline{E_t})\cap\Omega\big|
\le\big|u^{-1}(t)\cap \Omega\big|
=0.
\end{split}\end{equation*}
Together with~\eqref{jrugh694943} and \eqref{yth7a344356cP}, this implies that
\eqlab{\label{PAerq123PEoi}
	\big| \{u\geq t\}\triangle E_t\big|
=&\; 0,
	}
that is,
$E_t$ coincides with $\{u\geq t\}$ for almost every~$t\in \R$,
up to null sets.
As a result,
$$ \Per_s^L( \{u\geq t\} ,\Omega)=
\Per_s^L (E_t, \Omega).$$
For this reason,
employing the co-area formula in~\eqref{hhh} we have that
  	\begin{equation}\label{Pif83u4ehujPKSd5}
  		\frac12 [u]_{W^{s,1}(\Omega)} =\int_{-\infty}^\infty \Per_s^L( \{u\geq t\} ,\Omega) \, dt = \int_{-\infty}^\infty \Per_s^L (E_t, \Omega) \, dt.
  	\end{equation}
  	Also, using~\eqref{288},
  		\bgs{
  			|\varphi(y)| = \int_{0}^\infty \chi_{\{\varphi \geq t\} }(y) \, dt + \int_{-\infty}^0 (1-\chi_{\{\varphi \geq t\}}(y)) \, dt,
  		}
and therefore, recalling~\eqref{truntail} and exchanging integrals,
\begin{equation}\label{aiNREasdtr56}
\begin{split}&
\|\Tail_s(\varphi,\Omega_{\Theta d}\setminus \Omega, \cdot)\|_{L^1(\Omega)} \\=\;&	
\int_{\Omega} \left( \int_{\Omega_{\Theta d}\setminus \Omega} \frac{|\varphi(y)|} {|x-y|^{n+s}} \, dy \right) dx\\=\;&
\int_{0}^\infty \left[ \int_{\Omega}\left( \int_{\Omega_{\Theta d} \setminus \Omega} \frac{ \chi_{\{\varphi \geq t\}}(y)}{|x-y|^{n+s}} dy \right) dx \right] dt 
  			 +
  			\int_{-\infty}^0 \left[ \int_{\Omega} \left( \int_{\Omega_{\Theta d} \setminus \Omega} \frac{ |1-\chi_{\{ \varphi \geq t\} }(y)|}{|x-y|^{n+s}}  dy \right) dx \right] dt\\=\;&
  			\int_{0}^\infty  \|\Tail_s(\chi_{\{\varphi \geq t\}},\Omega_{\Theta d}\setminus \Omega, \cdot)\|_{L^1(\Omega)}\, dt + \int_{-\infty}^0   \|\Tail_s(1 - \chi_{\{\varphi \geq t\}},\Omega_{\Theta d}\setminus \Omega, \cdot)\|_{L^1(\Omega)}\, dt.
\end{split}\end{equation}
Furthermore, by Theorem~\ref{TUTT} and the $s$-minimality of the set~$E_t$, we know that~$\chi_{E_t}$
is an~$s$-minimal function, with external datum~$\chi_{\mathcal{E}_t}=\chi_{\{\varphi \geq t\}}$
in~$\Co\Omega$.
Consequently, using Theorem~\ref{thm51} with~$\lambda:=0$
and~$\lambda:=1$, we get that
\bgs{
& [\chi_{E_t}]_{W^{s,1}(\Omega)}\le \| \chi_{E_t}\|_{W^{s,1}(\Omega)}
	\le C\,
\|\Tail_s(\chi_{\{\varphi \geq t\}},\Omega_{\Theta d}\setminus \Omega, \cdot)\|_{L^1(\Omega)}, \quad \mbox{and }
	\\
	 &[ \chi_{E_t}]_{W^{s,1}(\Omega)}=
[1- \chi_{E_t}]_{W^{s,1}(\Omega)}
\le \|1- \chi_{E_t}\|_{W^{s,1}(\Omega)}
\le C\,
\|\Tail_s(1 - \chi_{\{\varphi \geq t\}},\Omega_{\Theta d}\setminus \Omega, \cdot)\|_{L^1(\Omega)}
}
for some~$C>0$. Hence, making use of~\eqref{PLa}
and~\eqref{aiNREasdtr56},
\bgs{
C\,\|\Tail_s(\varphi,\Omega_{\Theta d}\setminus \Omega, \cdot)\|_{L^1(\Omega)}
		\ge&\;
\int_{0}^\infty [\chi_{E_t}]_{W^{s,1}(\Omega)}\, dt + 	\int_{-\infty}^0[\chi_{E_t}]_{W^{s,1}(\Omega)}\, dt
	\\
	=&\;\int_{-\infty}^{+\infty}[\chi_{E_t}]_{W^{s,1}(\Omega)}\,dt
	\\
	=&\; 2
\int_{-\infty}^{+\infty}\Per_s^L(E_t,\Omega)\,dt.
}
For that reason and~\eqref{Pif83u4ehujPKSd5}, we have that
$$ C\,\|\Tail_s(\varphi,\Omega_{\Theta d}\setminus \Omega, \cdot)\|_{L^1(\Omega)}\ge
[u]_{W^{s,1}(\Omega)}.$$
This, together with~\eqref{PHIA}, gives that~$[u]_{W^{s,1}(\Omega)}<+\infty$,
thus proving that~$u\in \W_\varphi^s(\Omega)$.

Therefore, recalling also~\eqref{PAerq123PEoi},
we can apply Proposition~\ref{fgha} and obtain that~$u$ is an~$s$-minimal function, as desired.
\end{proof}

As a consequence of Theorem~\ref{exlev}, we can now prove the existence
result in Theorem~\ref{INTAL}, which is valid under an integrable control of the tail of the datum.

\begin{proof}[Proof of Theorem~\ref{INTAL}]
Let~$\Theta$ be as in Theorem~\ref{exlev}.
Given~$M>0$, we define
	\sys[\varphi_M:= ]
			{  & \min\left\{ M, \max \left\{ \varphi, -M\right\} \right\} &&  \mbox{ in } \quad  \Omega_{\Theta d} \setminus \Omega, 
			\\
			& \varphi &&  \mbox{ in } \quad  \Co  \Omega_{\Theta d}.
	}
	Notice that 
		\begin{equation}\label{7ujhed8uj3re8u}
			{|\varphi_M(y)|} \leq  {|\varphi (y)|}
		\qquad {\mbox{for all }}\;  y\in \Co \Omega,
		\end{equation}
and thus 
	\eqlab{\label{tailbM}
	\|\Tail_s(\varphi_M, \Omega_{\Theta d} \setminus \Omega; \cdot) \|_{L^1(\Omega)} \leq  \|\Tail_s(\varphi, \Omega_{\Theta d} \setminus \Omega; \cdot) \|_{L^1(\Omega)}<+\infty,
	}
thanks to~\eqref{Thbe2w4}.\\
In addition, according to Theorem~\ref{exlev}, there exists an
$s$-minimal function~$u_M \in \W^s_{\varphi_M}(\Omega)$.
Then, by Theorem~\ref{thm51},
	it holds that
	\bgs{
		\|u_M\|_{W^{s,1}(\Omega)} \leq C\,  \|\Tail_s(\varphi_M, \Omega_{\Theta d} \setminus \Omega; \cdot)\|_{L^1(\Omega)},
		} for some~$C>0$.
		This, together with~\eqref{tailbM}, says that $\|u_M\|_{W^{s,1}(\Omega)}$ is uniformly bounded, and thus, by compactness,  there exists~$u\in {W^{s,1}(\Omega)}$ such that, up to a subsequence, for  $M\to+\infty$,
	\begin{equation}\label{ED764y3r}
	u_M \longrightarrow u \qquad  \mbox{ in } \; L^1(\Omega) \quad \mbox{ and a.e. in } \;  \Omega
	.
	\end{equation} 
	It remains to prove that $u$ is an
$s$-minimal function. For this, we consider any competitor $v$, and denote 
	\sys[v_M: =]
		{& v && \mbox{ in } \; \Omega,
		\\
		&\varphi_M && \mbox{ in } \; \Co \Omega
	.}
	Then, since $u_M$ is an
$s$-minimal function and $v_M$ is a competitor for it,
\begin{equation}\label{093240476y} \iint_{Q(\Omega)} \left( |u_M(x)-u_M(y)| -|v_M(x)-v_M(y)|\right) \frac{ dx \, dy} {|x-y|^{n+s}}\le0.\end{equation}
We also point out that,
if~$x$, $y\in\Omega$,
\begin{equation}\label{00093240476y2}\begin{split}
&|u(x)-u(y)|-|u_M(x)-u_M(y)| -|v(x)-v(y)|+|v_M(x)-v_M(y)|\\
=\;&|u(x)-u(y)|-|u_M(x)-u_M(y)| .
\end{split}\end{equation}
Moreover, if~$x\in\Omega$ and~$y\in\Co\Omega_{\Theta d} $,
\begin{equation}\label{093240476y2}\begin{split}
&|u(x)-u(y)|-|u_M(x)-u_M(y)| -|v(x)-v(y)|+|v_M(x)-v_M(y)|\\
=\;&|u(x)-\varphi(y)|-|u_M(x)-\varphi(y)| -|v(x)-\varphi(y)|+|v(x)-\varphi(y)|\\
=\;&|u(x)-\varphi(y)|-|u_M(x)-\varphi(y)| \\
\le\;&|u(x)-u_M(x)| .
\end{split}\end{equation}
Similarly,
if~$x\in\Omega$ and~$y\in\Omega_{\Theta d} \setminus \Omega$,
\begin{equation*}\begin{split}
&|u(x)-u(y)|-|u_M(x)-u_M(y)| -|v(x)-v(y)|+|v_M(x)-v_M(y)|\\
=&\;|u(x)-u(y)|-|u_M(x)-u_M(y)| -|v(x)-\varphi(y)|+|v(x)-\varphi_M(y)|\\
\le&\;|u(x)-u(y)|-|u_M(x)-u_M(y)| +|\varphi(y)-\varphi_M(y)|.
\end{split}
\end{equation*}
This, \eqref{093240476y}, \eqref{00093240476y2} and~\eqref{093240476y2} give that
\begin{equation}\label{RISND:1}\begin{split}
\frac12
		\iint_{Q(\Omega)} & \left(|u(x)-u(y)|-|v(x)-v(y)|\right) \frac{ dx \, dy} {|x-y|^{n+s}} 
		\\
		 \leq  & \;\frac12 \iint_{Q(\Omega)} \left(|u(x)-u(y)|-|u_M(x)-u_M(y)| -|v(x)-v(y)|+|v_M(x)-v_M(y)|\right) \frac{ dx \, dy} {|x-y|^{n+s}} 
		 \\
		 \leq &\; 
		\frac12  \int_\Omega \int_\Omega \left(|u(x)-u(y)|-|u_M(x)-u_M(y)|\right) \frac{ dx \, dy} {|x-y|^{n+s}} 
		 \\
		 & \; +   \int_\Omega \int_{\Omega_{\Theta d} \setminus \Omega}   \left(|u(x)-u(y)|-|u_M(x)-u_M(y)|\right) \frac{ dx \, dy} {|x-y|^{n+s}} \\
		 &\;+  \int_\Omega \int_{\Omega_{\Theta d} \setminus \Omega}  \frac{|\varphi(y)-\varphi_M(y)| }{ |x-y|^{n+s}}   dx \, dy
 		  + \int_\Omega \int_{\Co \Omega_{\Theta d}}   \frac{|u(x)-u_M(x)|}{|x-y|^{n+s}}\, dx \, dy. 
			\end{split}\end{equation}
Now, by~\eqref{ED764y3r} and Fatou's Lemma, it holds that
\begin{equation}\label{RISND:2}
		 \int_\Omega \int_{\Omega} \frac{  |u(x)-u(y)| } {|x-y|^{n+s}} dx \, dy\leq  
		\liminf_{M\to+\infty} \int_\Omega  \int_\Omega \frac{ |u_M(x)-u_M(y)| } {|x-y|^{n+s}} dx \, dy
	\end{equation}			
We also observe that
	\begin{equation}\label{95i5832465} \varphi_M \longrightarrow \varphi \qquad \mbox{ as } \; M \to+\infty,\end{equation}
and therefore, using again Fatou's Lemma, 
\begin{equation}\label{RISND:3}
		 \int_\Omega \int_{\Omega_{\Theta d} \setminus \Omega} \frac{  |u(x)-u(y)| } {|x-y|^{n+s}} dx \, dy\leq  
		\liminf_{M\to+\infty}\int_\Omega \int_{\Omega_{\Theta d} \setminus \Omega } \frac{ |u_M(x)-u_M(y)| } {|x-y|^{n+s}} dx \, dy.
				\end{equation}
Now we observe that
$$ \frac{|\varphi(y)-\varphi_M(y)| }{ |x-y|^{n+s}}  \le
\frac{|\varphi(y)|+|\varphi_M(y)| }{ |x-y|^{n+s}} \leq \frac{2|\varphi(y)|}{ |x-y|^{n+s}}\in L^1\big(
\Omega\times(\Omega_{\Theta d} \setminus \Omega)\big),
$$
thanks to~\eqref{Thbe2w4} and~\eqref{7ujhed8uj3re8u}.
\\
{F}rom this, \eqref{95i5832465} and the Dominated Convergence Theorem, it follows that
\begin{equation}
	\label{limphi}
		\lim_{M \to+\infty} \int_{\Omega} \int_{\Omega_{\Theta d} \setminus \Omega} \frac{\big|\varphi(y)-\varphi_M(y)\big| }{|x-y|^{n+s}} \, dx \, dy =0.
	\end{equation}
		Furthermore, if~$x\in\Omega $ and~$y\in\Co \Omega_{\Theta d}$, we have that~$y\in\Co B_{\Theta d}(x)$, and accordingly
		\begin{equation*}
			 \int_\Omega \int_{\Co \Omega_{\Theta d}}    \frac{|u(x)-u_M(x)|} {|x-y|^{n+s}}  dx \, dy\leq
			 \int_\Omega \int_{\Co B_{\Theta d}(x)}    \frac{|u(x)-u_M(x)|} {|x-y|^{n+s}}  dx \, dy
			 \le\frac{C}{(\Theta d)^s}\, \|u-u_M\|_{L^1(\Omega)},  
		\end{equation*}
for some~$C>0$. Using this and~\eqref{ED764y3r}, we find that		\[
			 \lim_{M\to+\infty} \int_\Omega \int_{\Co \Omega_{\Theta d}}   \frac{|u(x)-u_M(x)|} {|x-y|^{n+s}}
			 \, dx \, dy =0.  
		\]
We plug this information and \eqref{RISND:2}, \eqref{RISND:3}, \eqref{limphi} into \eqref{RISND:1}
concluding that
$$ \frac12
		\iint_{Q(\Omega)} \left(|u(x)-u(y)|-|v(x)-v(y)|\right) \frac{ dx \, dy} {|x-y|^{n+s}} \le0.$$
This shows that $u$ is an $s$-minimal function and concludes the proof of the theorem.
\end{proof}

\section{Non-uniqueness of $s$-minimal sets, and proof of Theorem~\ref{NONUN}}\label{SJifngwhis}

In this section,
we consider an example which shows that the minimisers of $\tilde  \G$ need not be unique,
thus proving Theorem~\ref{NONUN}:
	
	\begin{proof}[Proof of Theorem~\ref{NONUN}]
Let $E \subset \R^2$ be an~$s$-minimal set
in $B_1$, with $E\setminus B_1= E_0$.
We can suppose that
\begin{equation}\label{UNI6374}
{\mbox{$E$ is the unique $s$-minimal set with exterior data $E_0$,}}
\end{equation}
otherwise we are done.
\\
We observe that
\begin{equation}\label{9uiq:9eifejvwhibc}
{\mbox{$\Co E$ is an~$s$-minimal set
in $B_1$ with respect to~$ \Co E_0$,}}\end{equation}
and we let~$\mathcal R_{\pi}$ be a ninety degree rotation (say, in the anti-clockwise
sense) and we define
\begin{equation}\label{HAS:200}
		E_\pi:= \mathcal R_{\pi} (\Co E).
		\end{equation}
By the rotation invariance of the fractional perimeter and~\eqref{9uiq:9eifejvwhibc}, we have that
\begin{equation}\label{HAS:2}
{\mbox{the set $E_\pi$ is~$s$-minimal in $B_1$ with respect to the datum~$E_0$.}}
\end{equation}
Then, by comparing~\eqref{HAS:2} with~\eqref{UNI6374}, it follows that
	\eqlab{ 
		\label{un}
		E=E_\pi.
		}
Now we claim that
\begin{equation}\label{not:inE}
0\not\in \partial E.
\end{equation}
Indeed, suppose for a contradiction
that $0\in \partial E$.
Then, by~\cite{SavV},
we know that $\partial E$ is $C^\infty$ around $0$,
and therefore 
\begin{equation}\label{IHSjsjd}
{\mbox{we can denote by $\nu$ the exterior unit
normal to $\partial E$ at $0$.}}\end{equation}
This gives that~$-\nu$ is the exterior normal at~$0$ of~$\Co E$ and therefore,
by~\eqref{HAS:200}, we have that~$\mathcal R_{\pi}(-\nu)$
is the exterior normal at~$0$ of~$E_\pi$. This, in light of~\eqref{un},
gives that~$\mathcal R_{\pi}(-\nu)$
is the exterior normal at~$0$ of~$E$.
Hence, by~\eqref{IHSjsjd},
we have that
\begin{equation}\label{HAN:92ei3902ryifghbndcnnb}
		\nu= \mathcal R_{\pi}(-\nu).
	\end{equation}
On the other hand, a ninety degree rotation sends a given vector to an
orthogonal one, hence
$$ -\nu\cdot \mathcal R_{\pi}(-\nu)=0.$$
This and~\eqref{HAN:92ei3902ryifghbndcnnb} yield that~$-1=-\nu\cdot\nu=0$,
which is a contradiction. This proves~\eqref{not:inE}.
\\
In view of~\eqref{not:inE}, we can suppose that
\begin{equation}\label{H78ANdcdjf8243t5}
{\mbox{$0$ lies
inside~$E$,}}\end{equation}
the case in which~$0$ lies
inside~$\Co E$ being analogous.
That is, by~\eqref{H78ANdcdjf8243t5},
we have that~$0$ lies outside~$\Co E$,
and therefore, by~\eqref{HAS:200}, it follows that
\begin{equation}\label{H78ANdcdjf8243t5:2}
{\mbox{$0$
lies outside~$E_\pi$.}}\end{equation}
By comparing~\eqref{H78ANdcdjf8243t5} and~\eqref{H78ANdcdjf8243t5:2},
we see that~$E\ne E_\pi$, and this is in contradiction
with~\eqref{un}.
\end{proof}

\appendix
\section{Additional remarks}\label{APPPES}

The following is an existence theorem, when we assume that the
``global tail'' is summable, according to the notation in~\eqref{coda}.  
\begin{theorem}\label{APP:TH:E}
If $\varphi \colon \Co \Omega \to \R$ is such that 
	\eqlab{ \label{tsss}
		\mathcal T_s(\varphi, \Omega) <+\infty,
	}
	then there exists  an
$s$-minimal function~$u\in \W_\varphi^s(\Omega)$. 
\end{theorem}
\begin{proof} We recall that the definition of~$\W_\varphi^s(\Omega)$ was
introduced in~\eqref{WsOM}.
The proof is carried out by using the direct method
of the calculus of variations. According to Lemma~\ref{TGAet}, in the hypothesis~\eqref{tsss}, looking for a minimiser of $\tilde \G$ is equivalent to looking for a minimiser of $\G$. Now,
since~$\G(u,\Omega)\geq 0$, we
can consider a minimising sequence~$u_k\in \W^s_\varphi(\Omega)$
for~$\G$. We
let 
	\[
		m:= \inf_{u \in \W_\varphi^s(\Omega)} \G(u,\Omega) = 
\liminf_{k\to+\infty} \G(u_k,\Omega).
	\]
	Then, for $k$ large enough, we have that
	\[
		[u_k]_{W^{s,1}(\Omega)} \leq m+1
	\]
	and by ~\cite[Lemma D.1.2]{tesilu}, also $\|u_k\|_{W^{s,1}(\Omega)}$ is uniformly bounded. As a consequence,
by compactness, there exists $u\in \W^s_\varphi(\Omega)$ such that
	\[
		u_k \longrightarrow u \quad \mbox{ for } k \to+\infty,  \quad \mbox{ in } L^1(\Omega) \: \; \mbox{ and a.e. in } \Omega.
	\] 
Accordingly, by Fatou's Lemma, we conclude that
	\[
		\G (u,\Omega) \leq \liminf_{k\to+\infty} \G(u_k,\Omega)=m,
	\]
	hence the thesis.
\end{proof}

For the sake of completeness, we give the next result.
\begin{lemma}\label{poi}
Let $\Omega \subset \Rn$ be an open set and let $f\in L^1_{\loc}(\Omega)$.
Let
	\[
		\Sigma:= \{  t\in\R \; | \;\; |\{f=t\}|>0 \}. 
		\]
		Then $\Sigma$ is at most numerable.	
\end{lemma}

\begin{proof}
Let $\Omega_k \subset \subset \Omega$ be bounded open sets such that
	\[
		\Omega_k \subset \subset \Omega_{k+1} , \qquad \mbox{ and } \qquad \bigcup_{k\in \N} \Omega_k =\Omega
		.\]
		Let
		\[
			\Sigma_k:= \left\{ t\in \R \setminus \{0\} \; |\;\; \left|\{f=t\} \cap \Omega_k \right|>0  \right\}.
			\]
Then,
			\[ 
			\Sigma \setminus \{0\}= \bigcup_{k\in \N} \Sigma_k
			.\]
			It is enough to check that $\Sigma_k$ is at most numerable to obtain the conclusion of the lemma. 
\\			
			In order to prove this, for every~$h$, $m\in\N$, let
			\[
			\Sigma_k^h:= \left\{ t\in \Sigma_k \; \big|\; |t|>\frac1h  \right\}\qquad {\mbox{and}}\qquad
			\Sigma_k^{h,m} = \left\{ t\in \Sigma_k^h \; \big|\;  \left|\{f=t\} \cap \Omega_k \right|>\frac1m \right\}
			,\]
			so that 
			\[ 
				\Sigma_k^h =\bigcup_{m\in \N} \Sigma_k^{h,m}\qquad {\mbox{and}}\qquad
				\Sigma_k=  \bigcup_{h\in \N} \Sigma_k^{h}.
			\]
			We now prove that $\Sigma_k^{h,m}$ is finite, which is enough to conclude. To this end, we show that
			\begin{equation}\label{Hj:8i2ejf7374:09r}
			\#   \Sigma_k^{h,m} \leq \lceil hm \|f\|_{L^1(\Omega_k)}   \rceil,
			\end{equation}
where we used the standard notation for the ``ceiling function''~$\lceil \cdot \rceil $.
			To prove~\eqref{Hj:8i2ejf7374:09r}, let $t_1, \dots, t_N$ be $N$ distinct elements of $\Sigma_k^{h,m}$. Notice that
			$ \{ f=t_i\} \cap \{ f=t_j\} =\emptyset$ if $i\neq j$, hence
		\begin{eqnarray*}&&
			\frac{N}{hm} < \sum_{i=1}^N |t_i| \; |\{ f=t_i\} \cap \Omega_k| 
			= \sum_{i=1}^N \int_{\{ f=t_i\} \cap \Omega_k} |f(x)| \, dx
			\\&&\qquad= \int_{\Omega_k}  \left(\sum_{i=1}^N \chi_{\{ f=t_i\} }(x)\right) |f(x)| \, dx
		\leq \|f\|_{L^1(\Omega_k)}.
		\end{eqnarray*}
			This concludes the proof of~\eqref{Hj:8i2ejf7374:09r}, as well as of the lemma.			
\end{proof}

\end{document}